\documentclass{amsart}

\usepackage{amssymb}
\usepackage{tkz-graph}
\tikzset{node distance=2cm, auto}

\newtheorem{them}{Theorem}[section]
\newtheorem{lem}[them]{Lemma}

\newtheorem{prop}[them]{Proposition}
\newtheorem{coro}[them]{Corollary}

\theoremstyle{definition}

\newtheorem{defn}[them]{Definition}
\newtheorem*{acknowledgements}{Acknowledgements}

\newcommand{\en}{\operatorname{End}}
\newcommand{\aut}{\operatorname{Aut}}
\newcommand{\rank}{\operatorname{rank}}
\renewcommand{\ker}{\operatorname{ker}}
\newcommand{\im}{\operatorname{im}}
\newcommand{\ig}{\operatorname{IG}}

\newcommand{\dee}{\mbox{$\mathcal{D}$}}
\newcommand{\jj}{\mbox{$\mathcal{J}$}}
\newcommand{\eh}{\mbox{$\mathcal{H}$}}
\newcommand{\ar}{\mbox{$\mathcal{R}$}}
\newcommand{\el}{\mbox{$\mathcal{L}$}}

\newcommand{\e}{\varepsilon}

\renewcommand{\epsilon}{\varepsilon}
\renewcommand{\a}{\alpha}
\renewcommand{\b}{\beta}
\newcommand{\leqel}{\mbox{$\leq _{\mathcal L }$}}
\newcommand{\leqar}{\mbox{$\leq _{\mathcal R}$}}

\begin{document}


\title[Free idempotent generated semigroups]%
{Free idempotent generated semigroups and \\ endomorphism monoids of free $G$-acts}


\author{Igor Dolinka}
\address{Department of Mathematics and Informatics, University of Novi Sad, Trg Dositeja Obra\-do\-vi\-\'{c}a 4,
21101 Novi Sad, Serbia} \email{dockie@dmi.uns.ac.rs}

\author{Victoria Gould}
\author{Dandan Yang}
\address{Department of Mathematics,University of York, Heslington, York YO10 5DD, UK}
\email{victoria.gould@york.ac.uk} \email{ddy501@york.ac.uk}


\dedicatory{Dedicated to the memory of Professor David Rees (1918--2013)}




\thanks{Research supported by Grant No.\ EP/I032312/1 of EPSRC, Grant No.\ 174019 of the Ministry of Education, Science,
and Technological Development of the Republic of Serbia, and by Grant No.\ 3606/2013 of the Secretariat of Science and
Technological Development of the Autonomous Province of Vojvodina.}




\begin{abstract}
The study of the free idempotent generated semigroup $\ig(E)$ over a biordered set $E$ began with the seminal work of
Nambooripad in the 1970s and has seen a recent revival with a number of new approaches, both geometric and
combinatorial.  Here we study $\ig(E)$ in the case $E$ is the  biordered set
 of a wreath product $G\wr \mathcal{T}_n$, where $G$ is a group and $\mathcal{T}_n$ is the full transformation monoid on $n$ elements. This wreath product is isomorphic to the endomorphism monoid of the free
$G$-act $\en F_n(G)$ on $n$ generators, and this provides us with a convenient approach.

We say that the {\em rank} of an element of $\en F_n(G)$ is the minimal number of (free) generators in its image.  Let
$\varepsilon=\varepsilon^2\in \en F_n(G).$ For rather straightforward reasons it is known that if $\rank \epsilon =n-1$
(respectively, $n$), then the maximal subgroup of $\ig(E)$ containing $\varepsilon$ is free (respectively, trivial). We
show that if $\rank \epsilon =r$ where $1\leq r\leq n-2$, then the maximal subgroup of $\ig(E)$ containing
$\varepsilon$ is isomorphic to that in $\en F_n(G)$ and hence to $G\wr \mathcal{S}_r$, where $\mathcal{S}_r$ is the
symmetric group on $r$ elements. We have previously shown this result in the case $ r=1$; however, for higher rank, a
more sophisticated approach is needed. Our current proof subsumes the case $r=1$ and thus provides another approach to showing that {\em any} group occurs as the maximal subgroup of some $\ig(E)$. On the other hand, varying $r$ again and  taking $G$ to be trivial, we obtain an alternative proof of the recent result
of Gray and Ru\v{s}kuc for the biordered set of idempotents of $\mathcal{T}_n.$
\end{abstract}


\maketitle

\vspace{-3mm}

\section{Introduction}\label{sec:intro}

Let $S$ be a semigroup and denote by $\langle E\rangle$ the subsemigroup of $S$ generated by the set of idempotents
$E=E(S)$ of $S$. If $S=\langle E\rangle$, then we say that $S$ is {\em idempotent generated}. The significance of such
semigroups was evident at an early stage: in 1966 Howie \cite{howie:1966} showed that every semigroup may be embedded
into one that is idempotent generated. To do so, he investigated the idempotent generated subsemigroups of
transformation monoids, showing in particular that for the full transformation monoid $\mathcal{T}_n$ on $n$ generators
(where $n$ is finite), the subsemigroup of singular transformations is idempotent generated. Erdos \cite{erdos:1967}
proved a corresponding `linearised' result, showing that the multiplicative semigroup of singular square matrices over
a field is idempotent generated (see also \cite{La}).  Fountain and Lewin \cite{fountainandlewin} subsumed these
results into the wider context of endomorphism monoids of independence algebras. We note here that sets and vector
spaces over division rings are examples of independence algebras, as are free (left) $G$-acts over a group $G$.

For any set of idempotents $E=E(S)$ there is a free object IG$(E)$ in the category of semigroups that are generated by $E$, given by the presentation
\[\ig(E)=\langle\, \overline{E}:  \bar{e}\bar{f}=\overline{ef},\, e,f\in E, \{ e,f\}
\cap \{ ef,fe\}\neq \varnothing\,\rangle,\] where here $\overline{E}=\{ \bar{e}:e\in E\}$.
We say that $\ig(E)$ is the {\em free idempotent generated semigroup over $E$}.  The relations in the presentation for
$\ig(E)$ correspond to taking {\em basic products} in $E$, that is, products between $e,f\in E$ where $e$ and $f$ are
comparable under one of the quasi-orders $\leqel$ or $\leqar$ defined on $S$. In fact, $E$ has an abstract
characterisation as a {\em biordered set}, that is, a partial algebra equipped with two quasi-orders satisfying certain
axioms. Biordered sets were introduced in 1979 by Nambooripad \cite{nambooripad:1979} in his seminal work on the
structure of regular semigroups, as was the notion of free idempotent generated semigroups $\ig(E)$. A celebrated
result of Easdown \cite{easdown:1985} shows every biordered set $E$ occurs as $E(S)$ for some semigroup $S$, hence we
lose nothing by assuming that our set of idempotents is of the form $E(S)$ for a semigroup $S$.

For any semigroup $S$ and any idempotent $e\in E(S)$, there is a maximal sub{\em group} of $S$ (that is, a subsemigroup
that is a group) having identity $e$; standard semigroup theory, briefly outlined in Section~\ref{sec:enF}, tells us
that this group is the equivalence class of $e$ under Green's relation $\eh$, usually denoted by $H_e$. The study of
maximal subgroups of $\ig(E)$ has a somewhat curious history. It was thought from the 1970s that all such groups would
be free (see, for example, \cite{mcelwee:2002,NP,Past}), but this conjecture was false.  The first published example of
a non-free group arising in this context  appeared in 2009 \cite{brittenham:2009}; an unpublished example of McElwee
from the earlier part of that decade was announced by Easdown in 2011 \cite{easdown:2011}. Also, the paper \cite{brittenham:2009} exhibited a strong relationship between maximal subgroups of $\ig(E)$ and algebraic topology: namely, it was shown that these groups are precisely fundamental groups of a  complex naturally arising from $S$ (called the {\it Graham-Houghton} complex of $S$). The 2012 paper of  Gray and
Ru\v{s}kuc \cite{gray:2012} showed that {\em any} group occurs as a maximal subgroup of some $\ig(E)$. Their approach
is to use existing machinery which affords presentations of maximal subgroups of semigroups, itself  developed by
Ru\v{s}kuc \cite{Ru}, refine this to give presentations of $\ig(E)$, and then, given a group $G$, to carefully choose a
biordered set $E$. Their techniques are significant and powerful, and have other consequences in \cite{gray:2012};  we
use their presentation in this article. However, to show that {\em any} group occurs as a maximal subgroup of $\ig(E)$,
a simple approach suffices \cite{gouldyang:2013}. We also note here that
 any group occurs as $\ig(E)$ for some {\em band}, that is, a semigroup of idempotents \cite{dolinkaRuskuc:2013}.

The approach of \cite{gouldyang:2013} is to consider the biordered set $E$ of non-identity idempotents of
a wreath product $G\wr \mathcal{T}_n$ or, equivalently, of the endomorphism monoid $\en F_n(G)$ of a free (left) $G$-act on
$n$ generators $\{ x_1,\hdots, x_n\}$ (see, for example, \cite[Theorem 6.8]{kkm:2000}). It is known that for a rank $r$ idempotent $\epsilon\in \en F_n(G)$ we have
$H_{\epsilon}\cong G\wr \mathcal{S}_r$. For a rank 1 idempotent $\epsilon\in E$,  the maximal subgroup
$H_{\overline{\epsilon}}$ is isomorphic to $H_{\epsilon}$ and hence to $G$ \cite{gouldyang:2013}. This followed a pattern established in  \cite{brittenham:2011} and \cite{gray:2012a} showing (respectively) that the multiplicative group of non-zero elements of any division ring $Q$ occurs as a maximal subgroup of a rank 1 idempotent in $\ig(E)$, where $E$ is the biordered set of idempotents of
$M_n(Q)$ for $n\geq 3$, and that any $\mathcal{S}_r$ occurs as a maximal subgroup of a rank $r$ idempotent in $\ig(F)$, where $F$ is the biordered set of idempotents of
a full transformation monoid
$\mathcal{T}_n$ for some $n\geq r+2$.  Another way of saying this is that  in both these cases, $H_{\bar{e}}\cong
H_e$ for the idempotent in question.

The aim of this current article is to extend the results of both \cite{gouldyang:2013} and \cite{gray:2012a} to  show that
for a rank $r$ idempotent $\epsilon\in \en F_n(G)$, with $1\leq r\leq n-2$, we have that $H_{\overline{\epsilon}}$ is
isomorphic to $H_{\epsilon}$ and hence to $G\wr\mathcal{S}_r$. We proceed  as follows. In Section~\ref{sec:enF} we
recall some basics of Green's relations on semigroups, and specific details concerning the structure of $\en F_n(G)$.
In Section~\ref{sec:pres} we show how to use the generic presentation for maximal subgroups given in \cite{gray:2012}
(restated here as Theorem 3.3) to obtain a presentation of $H_{\overline{\epsilon}}$; once these technicalities are in
place we sketch the strategy employed in the rest of the paper, and work our way through this in subsequent sections.
By the end of Section~\ref{sec:restricted} we are able to show that for $1\leq r\leq n/3$,
$H_{\overline{\epsilon}}\cong H_{\epsilon}$ (Theorem~\ref{thm:n/3}), a result corresponding to that in
\cite{dolinka:2012} for full linear monoids. To proceed further, we need more sophisticated analysis of the generators
of $H_{\overline{\epsilon}}$. Finally, in Section~\ref{sec:main}, we make use of the presentation of $G\wr
\mathcal{S}_r$ given in \cite{Lavers:1998} to show that we have the required result, namely that
$H_{\overline{\epsilon}}\cong H_{\epsilon}$, for any rank $r$ with $1\leq r\leq n-2$  (Theorem~\ref{thm:n-2}). It is
worth remarking that if $G$ is trivial, then $F_n(G)$ is essentially a set, so that $\en F_n(G)\cong \mathcal{T}_n$. We
are  therefore  able to recover, via  a rather different strategy, the main result of \cite{gray:2012a}.

\section{Preliminaries: Green's relations, and endomorphism monoids of free $G$-acts}\label{sec:enF}

In the course of studying the general structural features of semigroups, amongst the most basic tools are the five
equivalence relations that capture the ideal structure of a given semigroup $S$, called {\em Green's relations}. We
define for $a,b\in S$:
$$
a\;\ar\; b \Leftrightarrow aS^1=bS^1, \quad a\;\el\; b \Leftrightarrow S^1a=S^1b,\quad a\;\jj\; b \Leftrightarrow
S^1aS^1=S^1bS^1,
$$
where $S^1$ denotes $S$ with an identity element adjoined (unless $S$ already has one); hence, these three relations
record when two elements of $S$ generate the same right, left, and two-sided principal ideals, respectively.
Furthermore, we let $\eh=\ar\cap\el$, while $\dee=\ar\circ\el=\el\circ\ar$ is the join of the equivalences $\ar$ and
$\el$. As is well known, for finite semigroups we always have $\dee=\jj$, while in general the inclusions
$\eh\subseteq\ar,\el\subseteq\dee\subseteq\jj$ hold. The $\ar$-class of $a$ is usually denoted by $R_a$, and in a
similar fashion we use the notation $L_a,J_a,H_a$ and $D_a$.

It is also well known that a single $\dee$-class consists either entirely of regular elements, or of non-regular ones,
see \cite[Proposition 2.3.1]{howie:1995}. If $a\in S$ is regular, that is, $a=aba$ for some $b\in S$, then, for any
such $b$, it is clear that $ab, ba\in E(S)$ and $ab~\mathcal{R}~a~\mathcal{L}~ba.$ Therefore, regular $\dee$-classes
are precisely those containing idempotents, and for each idempotent $e$, the $\eh$-class $H_e$ is a group with identity
$e$. In fact, this is a maximal subgroup of the semigroup under consideration and all maximal subgroups arise in this
way.

There are natural orders on the set of $\ar$- and $\el$-classes of $S$, respectively, defined by $R_a\leq R_b$ if and
only if $aS^1\subseteq bS^1$, and $L_a\leq L_b$ if and only if $S^1a\subseteq S^1b$. In turn, these orders induce
quasi-orders $\leqar$ and $\leqel$ on $S$ (mentioned in the introduction), given by $a\;\leqar\; b$ if and only if
$R_a\leq R_b$, and $a\;\leqel\; b$ if and only if $L_a\leq L_b$. Further details of Green's relations and other
standard semigroup techniques may be found in \cite{howie:1995}.

Let $S$ be a semigroup with $E=E(S)$. The semigroup $\ig(E)$ defined in the introduction has some pleasant properties,
particularly with respect to Green's relations. It follows from the definition that the natural map
$\boldsymbol{\phi}:\ig(E)\rightarrow S$, given by $\bar{e}\boldsymbol{\phi}= e$, is a morphism onto $S'=\langle
E\rangle$. Since any morphism preserves $\el$-classes and $\ar$-classes, certainly so does $\boldsymbol{\phi}$. In
fact, the structure of the regular $\dee$-classes of $\ig(E)$ is closely related to that in $S$, as the following
result, taken from \cite{fitzgerald:1972,nambooripad:1979,easdown:1985,brittenham:2011,gray:2012}, illustrates.

\begin{prop}\label{prop:remarkable} Let $S,S',E=E(S), \ig(E)$ and $\boldsymbol{\phi}$ be as
above, and let $e\in E$.
\begin{enumerate}\item[(i)] The restriction of $\boldsymbol{\phi}$ to the set of idempotents of $\ig(E)$ is a bijection
onto E (and an isomorphism of biordered sets).
\item[(ii)] The morphism $\boldsymbol{\phi}$  induces a bijection between the set of
all $\ar$-classes (respectively $\el$-classes) in the $\mathcal{D}$-class of $\bar{e}$ in $\ig(E)$ and the
corresponding set in $\langle E\rangle$.
\item[(iii)] The restriction of $\boldsymbol{\phi}$ to $H_{\bar{e}}$ is a morphism onto $H_e$.
\end{enumerate}
\end{prop}

We now turn our attention to $ F_n(G)$ and the structure of its endomorphism monoid. The following notational convention will be useful: for any $u,v\in\mathbb{N}$ with $u\leq v$ we will denote $\{ u,u+1,\cdots, v-1,v\}$ and
$\{ u+1,\cdots, v-1\}$ by $[u,v]$ and $(u,v)$, respectively.

Let $G$ be a group, $n\in\mathbb{N},n\geq 3$, and let $F_n(G)=\bigcup_{i=1}^n Gx_{i}$ be a
rank $n$ free left $G$-act. We recall that, as a set,
$F_n(G)$ consists of the set of formal symbols $\{ gx_i:g\in G, i\in [1,n]\}$,
and we identify $x_i$ with $1x_i$, where $1$ is the identity of $G$. For any $g,h\in G$ and $1\leq i,j\leq n$ we have
that $gx_i=hx_j$ if and only if $g=h$ and $i=j$; the action of $G$ is given by
$g(hx_i)=(gh)x_i$. Let End $F_n(G)$ denote the endomorphism monoid of $F_n(G)$ (with composition left-to-right). The image of
$\alpha\in\en F_n(G) $ being a (free) $G$-subact, we can define the
{\em rank} of $\alpha$ to be the rank of $\im \alpha$.

Since $F_n(G)$ is an independence algebra, a direct application of Corollary 4.6 \cite{gould:1995} gives  a useful characterisation of Green's relations on End $F_n(G)$.

\begin{lem}\cite{gould:1995} \label{lem:green}
For any $\a, \b\in \en F_n(G)$, we have the following:

\begin{enumerate}\item[(i)] $\im \a =\im \b$ if and only if $\a \,\el\, \b$;

\item[(ii)] $\ker \a=\ker \b$ if and only if $\a\,\ar\, \b$;

\item[(iii)] $\rank \a= \rank \b$ if and only if $\a\,\dee\, \b$ if and only if $\a\,\mathcal{J}\, \b$.
\end{enumerate}
\end{lem}

 Each
$\a\in \en F_n(G)$ depends only on its
 action on the free generators $\{ x_i:i\in [1,n]\}$
and it is therefore convenient to write

\[x_{j}\alpha=w_{j}^{\alpha}x_{j\overline{\alpha}}\]
for $j\in [1,n]$.
This determines a function $\overline{\alpha}: [1,n]\longrightarrow [1,n]$ and an
element $\alpha_G=(w_1^\alpha,\hdots,w_n^{\alpha})\in G^n$.
It will frequently be convenient to express $\alpha$ as above as
\[\alpha=\begin{pmatrix}x_1&x_2&\hdots &x_n\\
w_{1}^{\alpha}x_{1\overline{\alpha}}&w_{2}^{\alpha}x_{2\overline{\alpha}}&\hdots&w_{n}^{\alpha}x_{n\overline{\alpha}}\end{pmatrix}.\]

\begin{them} \label{them:wreath}\cite{skornjakov:1979,kkm:2000} The function
\[\boldsymbol{\psi}: \en F_n(G)\mapsto G\wr T_n,\,\,\alpha\mapsto ( \alpha_G,\overline{\alpha})\]
is an isomorphism.
\end{them}

 Let $1\leq r\leq n$ and set $D_{r}=\{\alpha\in \en F_{n}(G)~|~ \mbox{rank} ~\alpha=r\}$, that is, $D_r$ is the $\dee$-class in $\en F_n(G)$ of any rank $r$ element.  We let $I$ and $\Lambda$  denote  the set of $\mathcal{R}$-classes and the set of $\mathcal{L}$-classes of $D_r$, respectively. Thus, $I$ is in bijective correspondence with  the set of kernels,
and $\Lambda$ with the set of images, of rank $r$ endomorphisms, respectively.  It is convenient to assume $I$ {\em is} the set of kernels of rank $r$ endomorphisms, and that
\[\Lambda=\{ (u_1,u_2,\hdots, u_r): 1\leq u_1<u_2<\hdots <u_r\leq n\} \subseteq [1,n]^r.\]
Thus  $\alpha\in R_i$ if and only if $\ker\alpha=i$ and
 $\alpha\in L_{(u_1,\hdots,u_r)}$ if and only if \[\im\alpha=Gx_{u_1}\cup Gx_{u_2}\cup\hdots
\cup Gx_{u_r}.\]

For every $i\in I$ and $\lambda\in \Lambda$, we put $H_{i\lambda}=R_{i}~\cap~L_{\lambda}$ so that $H_{i\lambda}$ is an $\mathcal{H}$-class of $D_r$. Where $H_{i \lambda}$ is a subgroup, we denote its identity by $\varepsilon_{i \lambda}$. It is notationally standard to use the same symbol $1$ to denote a selected element from both $I$ and $\Lambda$. Here we let
 $1= \langle (x_{1},x_{i}): r+1\leq i\leq n\rangle\in I$, that is, the congruence generated by $\{(x_{1},x_{i}): r+1\leq i\leq n\}$,  and $1=(1,2,\hdots, r)\in\Lambda$. Then  $H=H_{11}$ is a group $\mathcal{H}$-class in $D_{r}$, with identity $\varepsilon_{11}$.

 A typical element of $H$ looks like
\[\alpha=\begin{pmatrix} x_1&x_2&\hdots &x_r&x_{r+1}&\hdots &x_n\\
w^{\alpha}_1x_{1\overline{\alpha}}&w^{\alpha}_2x_{2\overline{\alpha}}&\hdots&w^{\alpha}_rx_{r\overline{\alpha}}&
w^{\alpha}_1x_{1\overline{\alpha}}&\hdots&w^{\alpha}_1x_{1\overline{\alpha}}\end{pmatrix}\]
which in view of the following lemma we may abbreviate without further remark to:
\[\alpha=\begin{pmatrix} x_1&x_2&\hdots &x_r\\
w^{\alpha}_1x_{1\overline{\alpha}}&w^{\alpha}_2x_{2\overline{\alpha}}&\hdots&w^{\alpha}_rx_{r\overline{\alpha}}\end{pmatrix},\]
where here we are regarding $\overline{\alpha}$ as an element of $\mathcal{S}_r.$ With this convention, it is clear that $\boldsymbol{\psi}|_{H}:H\longrightarrow G\wr \mathcal{S}_r$ is an isomorphism.

\begin{lem}\label{lem:autfr} The groups $H$ and $\aut F_r(G)$ are isomorphic under the map
\[\begin{pmatrix} x_1&x_2&\hdots &x_r&x_{r+1}&\hdots &x_n\\
w^{\alpha}_1x_{1\overline{\alpha}}&w^{\alpha}_2x_{2\overline{\alpha}}&\hdots&w^{\alpha}_rx_{r\overline{\alpha}}&
w^{\alpha}_1x_{1\overline{\alpha}}&\hdots&w^{\alpha}_1x_{1\overline{\alpha}}\end{pmatrix} \mapsto\]
\[ \begin{pmatrix} x_1&x_2&\hdots &x_r\\
w^{\alpha}_1x_{1\overline{\alpha}}&w^{\alpha}_2x_{2\overline{\alpha}}&\hdots&w^{\alpha}_rx_{r\overline{\alpha}}\end{pmatrix}.\] Consequently, $\aut F_r(G)\cong G\wr \mathcal{S}_r.$
\end{lem}

Under this convention, the identity $\e=\e_{11}$ of $H$ becomes
\[\e=\begin{pmatrix}x_1&\hdots& x_r\\ x_1&\hdots&x_r\end{pmatrix} .\]

 With the aim of specialising the presentation given in Theorem~\ref{them:pres}, we locate and distinguish elements in
 $H_{1\lambda}$ and $H_{i 1}$ for each $\lambda\in \Lambda$ and $i\in I$. For any equivalence relation $\tau$ on $[1,n]$ with $r$ classes, we write $\tau=\{B_{1}^{\tau}, \cdots, B_{r}^{\tau}\}$ (that is, we identify $\tau$ with the partition on $[1,n]$ that it induces).
 Let $l_{1}^{\tau},\cdots, l_{r}^{\tau}$ be the minimum elements of $B_{1}^{\tau}, \cdots, B_{r}^{\tau}$, respectively. Without loss of generality we suppose that  $l_{1}^{\tau}<\cdots<l_{r}^{\tau}$. Then $l_{1}^{\tau}=1$ and $l_{j}^{\tau}\geq j$, for any $j\in [2,r]$. Suppose now that $\alpha\in \en F_n(G)$ and rank $\alpha=r$, that is, $\alpha\in D_r$. Then $\ker \overline{\alpha}$ has $r$ equivalence classes.
Where $\tau=\ker{\overline{\alpha}}$ we simplify our notation by writing $B_j^{\ker\overline{\alpha}}=B_{j}^{\alpha}$ and
 $l^{\ker \overline{\alpha}}_j=
l^\alpha_j$. If there is no ambiguity over the choice of $\alpha$ we may simplify further to $B_j$ and $l_j$.

\begin{lem}\label{lem:kernels} Let $\alpha,\beta\in D_r$. Then $\ker \alpha=\ker\beta$ if and only if
$\ker\overline{\alpha}=\ker\overline{\beta}$ and for any $j\in [1,r]$ there exists $g_j\in G$ such that for any $k\in B^\alpha_j=B_j=B^\beta_j$, we have
$w_k^\alpha=w^\beta_kg_j$. Moreover, we can take $g_j=(w^\beta_{l_j})^{-1}w^\alpha_{l_j}$ for $j\in [1,r]$.
\end{lem}
\begin{proof} If $\ker\alpha=\ker\beta$, then clearly $\ker\overline{\alpha}=\ker\overline{\beta}$. Now for any $j\in [1,r]$
and $k\in B^\alpha_j=B^\beta_j$, we have that $((w^\alpha_{l_j})^{-1}x_{l_j})\alpha=
((w^\alpha_k)^{-1}x_k)\alpha$ and so $((w^\alpha_{l_j})^{-1}x_{l_j})\beta=
((w^\alpha_k)^{-1}x_k)\beta$, giving that $w^\alpha_k=w^\beta_k((w^\beta_{l_j})^{-1}w^\alpha_{l_j})$. We may thus take
$g_j=(w^\beta_{l_j})^{-1}w^\alpha_{l_j}$.

Conversely, suppose that
$\ker\overline{\alpha}=\ker\overline{\beta}$ (and has blocks $\{B_1, \cdots, B_r\}$) and for any $j\in [1,r]$ there exists $g_j\in G$ sastisfying the given condition. Let $ux_h,vx_k\in F_n(G)$. Then
\[\begin{array}{rcl}
(ux_h)\alpha=(vx_k)\alpha&\Leftrightarrow&  h,k\in B_j\mbox{ for some $j\in [1,r]$ and }
uw^\alpha_h=vw^\alpha_k\\
&\Leftrightarrow&   h,k\in B_j\mbox{ for some $j\in [1,r]$ and }
uw^\beta_hg_j=vw^\beta_kg_j\\
&\Leftrightarrow&   h,k\in B_j\mbox{ for some $j\in [1,r]$ and }
uw^\beta_h=vw^\beta_k\\
&\Leftrightarrow& (ux_h)\beta=(vx_k)\beta\end{array},\]
so that $\ker \alpha=\ker\beta$ as required.
\end{proof}

For the following, we denote by $P(n,r)$ the set of equivalence relations on $[1,n]$ having $r$ classes. Of course, $
|P(n,r)|=S(n,r)$, where $S(n,r)$ is a Stirling number of the second kind, but we shall not need that fact here.

\begin{coro}\label{coro:kernels} The map $\boldsymbol{\tau}: I\rightarrow  G^{n-r}\times P(n,r)$ given by
\[i\boldsymbol{\tau}=((w^\alpha_{2},\hdots, w^{\alpha}_{l_2-1},w^\alpha_{l_2+1},\hdots,w^\alpha_{l_r-1},
w^\alpha_{l_r+1},\hdots,
w^\alpha_n),\ker\overline{\alpha} )\]
where $\alpha\in R_i$ and $w^\alpha_{l_j}=1_G$, for all $j\in[1,r]$,
is a bijection.
\end{coro}
\begin{proof} For $i\in I$ choose $\beta\in R_i$ and then define
$\alpha\in\en F_n(G)$ by $x_k\alpha=w^\beta_k(w^\beta_{l_j})^{-1}x_j$, where $k\in B^{\beta}_j$. It is clear
from Lemma~\ref{lem:kernels} that $\ker\alpha=\ker\beta$ and so $\alpha\in R_i$. Now
$x_{l_j}\alpha=w^{\beta}_{l_j}(w^{\beta}_{l_j})^{-1}x_j=x_j$, so that $i\boldsymbol{\tau}$ is defined. An easy argument, again
from Lemma~\ref{lem:kernels}, gives that $\boldsymbol{\tau}$ is well defined and one-one.

For $\mu\in P(n,r)$ let
$\nu_{\mu}:[1,n]\rightarrow [1,r]$ be given by $k\nu_{\mu}=j$ where $k\in B^\mu_j$.
Now for $((h_1,\hdots, h_{n-r}),\mu)\in G^{n-r}\times P(n,r)$, define
\[\alpha=((1_G,h_1,\hdots, h_{l^\mu_2-2},I_G,h_{l^\mu_2-1},\hdots, h_{l^\mu_r-r},1_G,h_{l^\mu_r-r+1},\hdots, h_{n-r}),\nu_{\mu})\boldsymbol{\psi}^{-1},\] where $\boldsymbol{\psi}$ is defined as in Theorem 2.3.
It is clear that if $\alpha\in R_i$, then $i\boldsymbol{\tau} =((h_1,\hdots, h_{n-r}),\mu)$. Thus $\boldsymbol{\tau}$ is a bijection as required.
\end{proof}

\begin{coro}\label{Theta set}
Let $\Theta$ be the set defined by $$\Theta=\{\alpha\in D_r: x_{l_{j}^{\alpha}}\alpha=x_{j}, j\in[1,r]\}.$$ Then
 $\Theta$ is a  transversal of the $\eh$-classes of $L_1$.
\end{coro}
\begin{proof} Clearly, $\im \alpha =Gx_{1}\cup\cdots \cup Gx_{r}$, for any $\alpha\in \Theta$, and so that $\Theta$ is a subset of $L_{1}$.

Next, we show that for each $i\in I $, $|H_{i1}\cap \Theta|=1$.
Suppose that $\alpha,\beta\in \Theta$ and $\ker \alpha=\ker \beta$. Clearly  $\ker \overline{\alpha}=\ker\overline{\beta}$
and so $B^\alpha_j=B_j=B^\beta_j$ for any $j\in [1,r]$, and by definition of $\Theta$,
$w^\alpha_{l_j}=w^\beta_{l_j}=1_G$. It is then clear from Lemma~\ref{lem:kernels} that
for any $k\in B_j$ we have
\[x_k\alpha=w^\alpha_kx_j=w^\beta_kx_j=x_k\beta,\]
so that $\alpha=\beta$.

It only remains to show that for any $i\in I$ we have  $|H_{i1}\cap \Theta|\neq \varnothing$. By
Corollary~\ref{coro:kernels}, for  $i\in I$ we can find $\alpha\in R_i$  such that $w^\alpha_{l_j}=1_G$ for all $j\in
[1,r]$. Composing $\alpha$ with $\beta\in \en F_n(G)$ where $x_{l_j\overline{\alpha}}\beta=x_j$ for all $j\in [1,r]$
and $x_k\beta=x_1$ else, we clearly have that $\alpha\beta\in  H_{i1}\cap \Theta$.
\end{proof}

For each $i\in I$, we denote the unique element in $H_{i1}\cap\Theta$ by $\mathbf{r}_{i}$. Notice that
$\mathbf{r}_1=\e$.

On the other hand, for $\lambda=(u_1,u_2,\hdots,u_r)\in \Lambda$, we define
\begin{align*}
\mathbf{q}_{\lambda}=\mathbf{q}_{(u_{1},\cdots, u_{r})}&=\left(\begin{array}{cccccccccccccccccccc}
x_{1} & x_{2} & \cdots & x_{r} & x_{r+1} & \cdots & x_{n} \\
x_{u_{1}} & x_{u_{2}} & \cdots & x_{u_{r}} & x_{u_{1}} &\cdots & x_{u_{1}} \end{array}\right)\\
&= \left(\begin{array}{cccc} x_{1} & x_{2} & \cdots & x_{r}\\x_{u_{1}} & x_{u_{2}} & \cdots &
x_{u_{r}}\end{array}\right).
\end{align*}
It is easy to see that $\mathbf{q}_{\lambda}\in H_{1\lambda}$, as $\ker \mathbf{q}_{\lambda}= \langle (x_{1},x_{i}):
r+1\leq i\leq n\rangle$. In particular, we have \[\mathbf{q}_1=\mathbf{q}_{(1,2, \cdots,
r)}=\left(\begin{array}{ccccccccc}
x_{1} & x_{2} & \cdots & x_{r} \\
x_{1} & x_{2} & \cdots & x_{r}  \end{array}\right)=\e.\]

At this point we invoke once more a little standard semigroup theory. Let $K$ be a group, let $J,\Gamma$ be non-empty
sets and let $M=(m_{\gamma j})$ be a $\Gamma\times J$ matrix with entries from $K\cup \{0\}$ (where $0$ is a new
symbol), with the property that every row and column of $M$ contains at least one entry from $K$. A {\em Rees matrix
semigroup} $\mathcal{M}^0=\mathcal{M}$$^{0}$$(K; J, \Gamma; M)$ has underlying set
\[(J\times K\times \Gamma)\cup \{ 0\}\]
with binary operation given by
\[(j,a,\lambda)(k,b,\mu)=(j,am_{\lambda k}b,\mu)\mbox{ if }m_{\lambda k}\neq 0,\]
all other products being $0$.

By \cite[Theorem 4.9]{gould:1995} if we put
\[D_r^0=D_r\cup \{ 0\}\]
and define a binary operation by
\[\alpha\cdot \beta=\left\{ \begin{array}{ll} \alpha\beta&\mbox{if }\alpha,\beta\in D_r\mbox{ and }\rank \alpha\beta=r\\
0&\mbox{ else}\end{array} \right. ,\] then $D_r^0$ is a semigroup under $\cdot$ which is completely 0-simple. We do not
need to give the specifics of what the latter property entails, since,  by the  Rees Theorem (see \cite[Chapter
III]{howie:1995}),  $D_{r}^{0}$ is isomorphic to  $\mathcal{M}^0=\mathcal{M}$$^{0}$$(H; I, \Lambda; P)$, where
$P=(\mathbf{p}_{\lambda i})$ and $\mathbf{p}_{\lambda i}=(\mathbf{q}_{\lambda}\mathbf{r}_{i})$ if $\rank
\mathbf{q}_{\lambda}\mathbf{r}_{i}=r$, and is $0$ else. Our choice of $P$ will allow us at crucial points to modify
the presentation given in Theorem~\ref{them:pres}.

\section{Presentation of maximal subgroups of $\ig(E)$} \label{sec:pres}

Let $E$ be a biordered set; from \cite{easdown:1985} we can assume that $E=E(S)$ for some semigroup $S$. An {\it
$E$-square} is a sequence $(e,f,g,h,e)$ of elements of $E$ with $e~\mathcal{R}$
$~f~\mathcal{L}$ $~g~\mathcal{R}$ $~h~\mathcal{L}$ $~e$. We draw such an $E$-square as $\begin{bmatrix} e&f\\
h&g\end{bmatrix}$. The following results are folklore (cf. \cite{gouldyang:2013}).

\begin{lem} \label{lem:esquarerectband} The elements of
an $E$-square $\begin{bmatrix} e&f\\ h&g\end{bmatrix}$  form a rectangular band
(within $S$) if and only if one (equivalently, all) of the following four equalities holds: $eg=f$, $ge=h$, $fh=e$ or $hf=g$.
\end{lem}

\begin{lem}\label{lem:sandwich}
Let $\mathcal{M}^0=\mathcal{M}$$^{0}$$(K; J, \Gamma; M)$ be a Rees matrix semigroup over a group $K$ with sandwich matrix
$M=(m_{\lambda j})$. For any $j\in J, \lambda\in \Gamma$ write $e_{j\lambda}$ for the idempotent $(j,m_{\lambda j}^{-1},\lambda)$. Then an
$E$-square $\begin{bmatrix} e_{i\lambda}&e_{i\mu}\\ e_{j\lambda}&e_{j\mu}\end{bmatrix}$  is a rectangular band if and only if
$m_{\lambda i}^{-1}m_{\lambda j}=m_{\mu i}^{-1}m_{\mu j}$.
\end{lem}

An $E$-square $(e,f,g,h,e)$ is {\it singular} if, in addition, there exists $k\in E$ such that either:
\[
\left\{\begin{array}{ll}
ek=e,\, fk=f, \,ke=h,\,  kf=g \mbox{ or}\\
ke=e,\, kh=h,\, ek=f,\, hk=g.
\end{array}\right.
\]
We call a singular square for which the first condition holds an {\it up-down  singular square}, and that satisfying  the second condition a {\it left-right singular square}.

For $e\in E$ we let $\overline{H}$ be the maximal subgroup of $\overline{e}$ in $\ig(E)$, (that is,
$\overline{H}=H_{\overline{e}}$).
We now recall the recipe for obtaining a presentation for $\overline{H}$  obtained by  Gray and  Ru\v{s}kuc \cite{gray:2012}; for further details, we refer the reader to that article.

 We use $J$ and $\Gamma$ to denote  the set of $\mathcal{R}$-classes and the set of $\mathcal{L}$-classes, respectively, in the $\mathcal{D}$-class $\overline{D}=D_{\overline{e}}$ of $\overline{e}$ in $\ig(E)$. In view of Proposition~\ref{prop:remarkable}, $J$ and $\Gamma$ also label
 the set of $\mathcal{R}$-classes and the set of $\mathcal{L}$-classes, respectively, in the $\mathcal{D}$-class $D=D_e$ of $e$ in $S$. For every $i\in J$ and $\lambda\in \Gamma$, let $\overline{H}_{i\lambda}$ and $H_{i\lambda}$ denote, respectively, the $\eh$-class corresponding to the intersection of the $\ar$-class indexed by $i$ and the $\el$-class indexed by $\lambda$ in
 $\ig(E)$, respectively $S$,
 so that $\overline{H}_{i\lambda}$ and $H_{i\lambda}$ are $\mathcal{H}$-classes of $\overline{D}$
 and $D$, respectively. Where $\overline{H}_{i\lambda}$ (equivalently, $H_{i\lambda}$) contains an idempotent, we denote it by $\overline{e_{i\lambda}}$ (respectively, $e_{i\lambda}$).
 Without loss of generality we assume $1\in J\cap \Gamma$ and $\overline{e}=\overline{e_{11}}\in\overline{H}_{11}=\overline{H}$, so that $e=e_{11}\in H_{11}=H$.
For each $\lambda\in \Gamma,$ we abbreviate $\overline{H}_{1\lambda}$ by $\overline{H}_{\lambda}$,
and  ${H}_{1\lambda}$ by ${H}_{\lambda}$ and so, $\overline{H}_{1}=\overline{H}$ and $H_1=H$.

Let $\overline{h}_{\lambda}$ be an element in $\overline{ E}^{*}$ such that $\overline{H}_{1}\overline{h}_{\lambda}=
\overline{H}_{\lambda}$, for each $\lambda\in \Gamma$. The reader should be aware that this is a point where we are most certainly abusing notation: whereas $\overline{h}_{\lambda}$ lies in the free monoid on $\overline{E}$, by writing
$\overline{H}_{1}\overline{h}_{\lambda}=
\overline{H}_{\lambda}$ we mean that the image of $\overline{h}_{\lambda}$ under the natural map that takes $\overline{E}^*$ to (right translations in) the full transformation monoid on $\ig(E)$ yields
$\overline{H}_{1}\overline{h}_{\lambda}=
\overline{H}_{\lambda}$.
In fact, it follows from Proposition~\ref{prop:remarkable} that the action of any generator
$\overline{f}\in \overline{E}$ on an $\mathcal{H}$-class contained in the $\mathcal{R}$-class of $\overline{e}$ in $\ig(E)$ is equivalent to the action of $f$ on the corresponding $\mathcal{H}$-class  in the original semigroup $S$. Thus
$\overline{H}_{1}\overline{h}_{\lambda}=
\overline{H}_{\lambda}$ in $\ig(E)$  is equivalent to the corresponding statement
 $H_1h_{\lambda}=H_{\lambda}$  for $S$, where $h_{\lambda}$ is the image of $\overline{h}_{\lambda}$ under the natural map to
 $\langle E\rangle ^1$.

We say that $\{\overline{h}_{\lambda}~|~\lambda\in \Gamma\}$ forms a {\it{Schreier system of representatives}}, if every prefix of $\overline{h}_{\lambda}$ (including the empty word) is equal to some $\overline{h}_{\mu}$, where $\mu\in \Gamma$.
Notice that the condition that $ \overline{h}_{\lambda}\overline{e}_{i\mu}=\overline{h}_{\mu}$ is equivalent to saying that $\overline{h}_{\lambda}\overline{e}_{i\mu}$
lies in the Schreier system.

Define $K=\{(i,\lambda)\in J\times \Gamma: H_{i\lambda} \mbox{~is a group~} \mathcal{H}\mbox{-class}\}$.
Since $D_e$ is regular, for each $i\in J$ we can find and fix an element
$\omega(i)\in\Gamma$ such that $(i,\omega(i))\in K$, so that
$\omega:J\rightarrow \Gamma$ is a function. Again for convenience we take $\omega(1)=1$.

\begin{them}\label{them:pres}\cite{gray:2012} The maximal subgroup $\overline{H}$ of $\overline{e}$ in $\ig(E)$ is defined by the presentation \[{\mathcal P}=\langle F:\Sigma\rangle\] with generators:
\[F = \{f_{i,\lambda}:~~(i,\lambda)\in K\}\]
and  defining relations $\Sigma$:

$(R1)~ f_{i,\lambda}=f_{i,\mu}\ \ \  ( \overline{h}_{\lambda}\overline{e}_{i\mu}=\overline{h}_{\mu})$;

$(R2)~f_{i,\omega(i)}=1\ \ \ (i\in J)$;

$(R3)~f_{i,\lambda}^{-1}f_{i,\mu}=f_{k,\lambda}^{-1}f_{k,\mu}\ \ \ \bigg(\begin{bmatrix} e_{i\lambda}&e_{i\mu}\\ e_{k\lambda}&e_{k\mu}\end{bmatrix} \mbox{~is a singular square}\bigg)$.
\end{them}

{\it For the remainder of this paper, $E$ will denote
 $E(\en F_n(G))$.} In addition, for the sake of notational convenience, we now observe the accepted convention of
 dropping the overline notation for elements of $\overline{E}^*$. In particular, idempotents of $\ig(E)$ carry the same notation as those of $\en F_n(G)$; the context should hopefully prevent confusion.

 In order to specialise the above presentation to $E$, our first step is
 to identify the singular squares.

\begin{lem}\label{lem:ss}
An $E$-square $\begin{bmatrix} \gamma&\delta\\ \xi&\nu\end{bmatrix}$ is singular if and only if
$\{ \gamma,\delta,\nu,\xi\}$ is   a rectangular band.
\end{lem}

\begin{proof}
The proof of necessity is standard. We only need to show the sufficiency. Let
$\{ \gamma,\delta,\nu,\xi\}$  be a rectangular band so that $\gamma\nu=\delta,\nu\gamma=\xi,\delta\xi=\gamma$ and $\xi\delta=\nu$. Suppose $\im \gamma$ = $\im \xi$ = $\langle x_{m}\rangle_{m\in M}$  and   $\im \delta$ =  $\im \nu$ = $\langle x_{n}\rangle_{n\in N}$, where $|M|=|N|=r$. Put $L=M\cup N$. Define a mapping $\theta\in \en F_{n}(G)$ by
\[ x_{i}\theta = \left\{ \begin{array}{lll}
         x_{i} & \mbox{if $i\in L$};\\
         x_{i}\nu & \mbox{else}.\end{array} \right. \]
Since $\im \theta=\langle x_{l}\rangle_{l\in L}$ and for each $l\in L$, $x_{l}\theta=x_{l}$, we see that $\theta$ is an idempotent. It is also clear that $\gamma\theta=\gamma$ and $\delta\theta=\delta$, as $\im \gamma\cup \im \delta\subseteq$ $\im \theta$.

Next, we will show $\theta\gamma=\xi$. If $i\in M$, then $x_{i}\theta\gamma=x_{i}\gamma=x_{i}=x_{i}\xi$. If $i\in N$, but $i\notin M$, then $x_{i}\theta\gamma=x_{i}\gamma=x_{i}\nu\gamma=x_{i}\xi$. If $i\notin L$, then $x_{i}\theta\gamma=x_{i}\nu\gamma=x_{i}\xi$. So, $\theta\gamma=\xi$. For the remaining equality $\theta\delta=\nu$ required in the definition of a singular square, observe that, for each $i\in N$, $x_{i}\theta\delta=x_{i}=x_{i}\nu$.
On the other hand if $i\in M$, then
$x_i\theta\delta= x_i\delta=x_i\gamma\nu=x_i\nu$, since $\delta=\gamma\nu$ by assumption.  For $i\notin L$ we have $x_{i}\theta\delta=x_{i}\nu\delta=x_{i}\nu$,
since $\delta\, \el\,  \nu$.

We have proved that $\gamma\theta=\gamma$, $\delta\theta=\delta$, $\theta\gamma=\xi$ and $\theta\delta=\nu$, so that $\begin{bmatrix} \gamma&\delta\\ \xi&\nu\end{bmatrix}$ is an up-down singular square.
\end{proof}

The proof of Lemma \ref{lem:ss} shows the following:

\begin{coro}\label{cor:updown}
An $E$-square is singular if and only if it is an up-down singular square.
\end{coro}

The next corollary is immediate from Lemmas~\ref{lem:sandwich} and \ref{lem:ss}.

\begin{coro}\label{lem:ssindr}
Let $P=(\mathbf{p}_{\lambda i})$ be the sandwich matrix of any completely 0-simple semigroup isomorphic to
$D_{r}^{0}$. Then (R3) in Theorem~\ref{them:pres} can be restated as:

$ (R3) ~ f_{i,\lambda}^{-1}f_{i,\mu}=f_{k,\lambda}^{-1}f_{k,\mu}  \ \ \ (\mathbf{p}_{\lambda i}^{-1}\mathbf{p}_{\lambda
k}=\mathbf{p}_{\mu i}^{-1}\mathbf{p}_{\mu k})$.
\end{coro}

We focus on the idempotent $\epsilon=\epsilon_{11}$ of Section~\ref{sec:enF}. For the presentation ${\mathcal
P}=\langle F:\Sigma\rangle$ for our particular $\overline{H}=H_{\overline{\epsilon}}$,  we must define a
Schreier system of words $\{\mathbf{h}_{\lambda}:\lambda\in\Lambda\}$. In this instance, we can do so inductively,
using the restriction of the lexicographic order on $[1,n]^r$ to $\Lambda$. Recall that we are using the same notation
for $\mathbf{h}_{\lambda}\in E^*$ and its image under the natural morphism to the set of right translations of $\ig(E)$ and of $\en
F_n(G)$.

First, we define $\mathbf{h}_{(1,2, \cdots, r)}=1$, the empty word in $E^*$. Now let $(u_1,u_2,\hdots, u_r)\in \Lambda$
with $(1,2, \cdots, r)<(u_1,u_2,\hdots, u_r)$, and assume for induction that $\mathbf{h}_{(v_1,v_2,\hdots,v_r)}$ has
been defined for all $(v_1,v_2,\hdots,v_r)<(u_1,u_2,\hdots,u_r)$.
 Taking $u_{0}=0$ there must exist some $j\in [1,r]$ such that $u_{j}-u_{j-1}>1$.
Letting $i$ be largest such that $u_{i}-u_{i-1}>1$ observe that
\[(u_1,\hdots,u_{i-1},u_i-1,u_{i+1},\hdots,u_r)<(u_1,u_2,\hdots,u_r).\] We now define
\[\mathbf{h}_{(u_{1},\cdots, u_{r})}=\mathbf{h}_{(u_{1},\cdots, u_{i-1}, u_{i}-1, u_{i+1},\cdots, u_{r})}\alpha_{(u_{1},\cdots, u_{r})},\] where
\begin{align*}
\alpha_{(u_{1},\cdots, u_{r})}=\left(\begin{array}{cccccccccccccc} x_{1} & \cdots & x_{u_{1}} & x_{u_{1}+1} &\cdots &
x_{u_{2}} & \cdots  & x_{u_{r-1}+1} & \cdots \\
x_{u_{1}} & \cdots & x_{u_{1}} & x_{u_{2}} & \cdots & x_{u_{2}} & \cdots  & x_{u_{r}} & \cdots \\ \end{array}\right. & \\
\left.\begin{array}{ccccc} \cdots & x_{u_{r}} & x_{u_r+1}& \cdots & x_{n} \\
\cdots & x_{u_{r}} & x_{u_r}&\cdots & x_{u_{r}} \\ \end{array}\right) ; &
\end{align*}
 notice that $\alpha_{(u_{1},\cdots, u_{r})}=\e_{l(u_1,\hdots, u_r)}$
for some $l\in I$, where  $\e_{l(u_1,\hdots, u_r)}$ is the idempotent contained in the $\mathcal{H}$-class with kernel indexed by $l$ and image indexed by $(u_1,\hdots, u_r).$

\begin{lem}\label{lem:righttrans}  For all $(u_1,\hdots,u_r)\in \Lambda$ we have
$\e \mathbf{h}_{(u_1,\hdots,u_r)}=\mathbf{q}_{(u_1,\hdots,u_r)}$. Hence right translation by
$\mathbf{h}_{(u_{1},\cdots, u_{r})}$ induces a bijection from $L_{(1,\cdots,r)}$ onto $L_{(u_{1},\cdots, u_{r})}$ in
both $\en F_n(G)$ and $\ig(E)$.
\end{lem}
\begin{proof}
We prove  by induction on $(u_1,\hdots,u_r)$ that $ \e \mathbf{h}_{(u_1,\hdots,u_r)}=\mathbf{q}_{(u_1,\hdots,u_r)}$.
Clearly the statement is true for $(u_1,\hdots,u_r)=(1,\hdots,r)$. Suppose now result is true for all
$(v_1,\hdots,v_r)<(u_1,\hdots,u_r)$, so that
$$
\e \mathbf{h}_{(u_{1},\cdots, u_{i-1}, u_{i}-1, u_{i+1},\cdots, u_{r})}=\mathbf{q}_{(u_{1},\cdots, u_{i-1}, u_{i}-1,
u_{i+1},\cdots, u_{r})}.
$$
Since $x_{u_j}\alpha=x_{u_j}$ for all $j\in \{ 1,\hdots,r\}$ and $x_{u_i-1}\alpha=x_{u_i}$, it follows that
\begin{align*}
\e \mathbf{h}_{(u_1,\hdots,u_r)}&=\e \mathbf{h}_{(u_{1},\cdots, u_{i-1}, u_{i}-1, u_{i+1},\cdots, u_{r})}
\alpha_{(u_1,\hdots,u_r)}\\ &=\mathbf{q}_{(u_{1},\cdots, u_{i-1}, u_{i}-1, u_{i+1},\cdots,
u_{r})}\alpha_{(u_1,\hdots,u_r)}= \mathbf{q}_{(u_1,\hdots,u_r)} \end{align*} as required.

Since by definition, $\mathbf{q}_{(u_1,\hdots,u_r)} \in L_{(u_{1},\cdots, u_{r})}$, the result for $\en F_n(G)$ follows
from Green's Lemma (see, for example, \cite[Chapter II]{howie:1995}), and that for $\ig(E)$ by the comments in
Section~\ref{sec:pres}, namely, that the action of any generators $\overline{f}\in \overline{E}$ on an $\mathcal{H}$-class contained in the $\mathcal{R}$-class $\overline{e}$ in $\ig(E)$ is equivalent to the action of $f$ on the corresponding $\mathcal{H}$-class in $\en F_n(G)$. \end{proof}

It is a consequence of Lemma~\ref{lem:righttrans}  that $\{\mathbf{h}_{\lambda}:\lambda\in \Lambda\}$ forms the
required Schreier system for  a presentation ${\mathcal{P}}$ for $\overline{H}$. It remains to define the function
$\omega$: we do so by setting
 $\omega(i)=(l^{\mathbf{r}_i}_1,l^{\mathbf{r}_i}_2,\hdots,l^{\mathbf{r}_i}_r)=
(1,l^{\mathbf{r}_i}_2,\hdots,l^{\mathbf{r}_i}_r)$ for each $i\in I$.   Note that for any $i\in I$ we have
$\mathbf{q}_{\omega(i)}\mathbf{r}_i=\e$, i.e. $\mathbf{p}_{\omega(i),i}=\e$.

\begin{defn}\label{defn:ourpres}  Let
$\mathcal{P}=\langle F:\Sigma\rangle$ be the presentation of  $\overline{H}$ as in Theorem~\ref{them:pres}, where
$\omega$ and $\{ \mathbf{h}_{\lambda}:\lambda\in\Lambda\}$ are given as above.
\end{defn}

Without loss of generality, we  assume
that $\overline{H}$ {\em is} the group with presentation ${\mathcal P}$.

In later parts of this work we will be considering for a non-zero entry $\phi\in P$, which $i\in I, \lambda\in \Lambda$
yield $\phi=\mathbf{p}_{\lambda i}$. For this and other purposes it is convenient to define the notion of { district}.
For $i\in I$ we say that $\mathbf{r}_{i}$ lies in {\it district $(l_{1}^{\mathbf{r}_{i}}, l_{2}^{\mathbf{r}_{i}},
\cdots, l_{r}^{\mathbf{r}_{i}})$} (of course, $1=l_{1}^{\mathbf{r}_{i}}$). Note that the district of  $\bold{r}_{i}$ is determined by the kernel of the transformation $\overline{\bold{r}_{i}},$ and lying in the same district induces a partition of  $\Theta$.

Let us run an example, with $n=9$ and $r=3$ we  consider the following two partitions: $$P_1=\{\{1,2,8\},\{3,4,7\}, \{5,6,9\}\}, P_2=\{\{1,4,6\},\{3,2\},\{5,8,9,7\}\}.$$ Then for each of these partitions if we take the minimal entries in each class in both cases we get $1<3<5,$ so these two partitions determine the same district.

 The next lemma follows immediately from the definition of $\mathbf{r}_i,i\in I$.

\begin{lem}\label{property of r_{i}}
For any $i\in I$, if $\mathbf{r}_{i}$ lies in district $(1,l_{2},\cdots,l_{r})$, then $l_{s}\geq s$ for all $s\in
[1,r]$. Moreover, for $k\in[1,n]$, if  $x_{k}\mathbf{r}_{i}=ax_{j}$, then $k\geq l_j$, with $k>l_j$ if $a\neq 1_G$.
\end{lem}

\begin{proof}
It is clear that $l_s\geq s$ for all $s\in [1,r].$ If $x_{k}\bold{r}_{i}=ax_{j}$, then $k\geq l_j$ because $k$ and $l_j$ belong to the same kernel class of $\overline{\bold{r}_i}.$ If $a\neq 1_G$, then $x_k\bold{r}_i\neq x_j$ so $k\neq l_j$ and hence $k>l_j.$
\end{proof}


We pause to consider which elements of $H$ can occur as an entry $\phi$ of $P$; with abuse of terminology, we will say
that $\phi\in P$. As indicated before Lemma~\ref{lem:autfr}, we can write $\phi\in H$ as
\[\phi=\begin{pmatrix}x_1&x_2&\hdots&x_r\\a_1x_{1\overline{\phi}}&a_2x_{2\overline{\phi}}&\hdots &a_rx_{r\overline{\phi}}\end{pmatrix}\]
where $\overline{\phi}\in\mathcal{S}_r$ and $(a_1,\hdots,a_r)\in G^r$. If $\phi=\mathbf{p}_{\lambda i}\in P$, where
$\lambda=(u_1,\hdots, u_r)$ and $\mathbf{r}_i$ lies in district $(l_1,\hdots, l_r)$, then the $u_j$s and $l_k$s are
constrained by
\[1=l_1<l_2<\hdots <l_r,\,\, u_1<u_2<\hdots<u_r,\]
\[l_{j\overline{\phi}}\leq u_j\mbox{ for all }j\in [1,r]\mbox{ with }l_{j\overline{\phi}}<u_j\mbox{ if }a_j\neq 1_G,\]
and
\[l_k=u_j \mbox{~implies~} k=j\overline{\phi} \mbox{~and~} a_j=1_G \mbox{~for~all~} k, j\in [1,r].\]
Conversely, if these constraints are satisfied by $l_1,\hdots, l_r,u_1,\hdots, u_r\in [1,n]$ with respect to some
$\overline{\phi}\in \mathcal{S}_r$ and $(a_1,\hdots,a_r)\in G^r$, then it is easy to see that if $\xi\in\en F_n(G)$ is defined by
\[x_{l_k}\xi=x_k,\, x_{u_k}\xi=a_kx_{k\overline{\phi}},\, k\in [1,r]\]
and
\[x_j\xi=x_1\mbox{ for }j\notin \{ l_1,\hdots,l_r,u_1,\hdots,u_r\},\]
then $\xi=r_i$ for some $i\in I$, where $r_i$ lies in district $(l_1, l_2, \cdots, l_r).$
Clearly, $\mathbf{p}_{\lambda i}=\phi$.

\begin{lem}\label{lem:all} If $|G|>1$ then every element of $H$ occurs as an entry in $P$ if and only if $2r\leq n$.
If $|G|=1$ then every element of $H$ occurs as an entry in $P$ if and only if $2r\leq n+1$.
\end{lem}
\begin{proof} Suppose first that $|G|>1$. If $2r\leq n$, then given any
$$\alpha=\begin{pmatrix}x_1&x_2&\hdots&x_r\\a_1x_{1\overline{\alpha}}&a_2x_{2\overline{\alpha}}&\hdots &a_rx_{r\overline{\alpha}}\end{pmatrix}$$
in $ H$, we can take
$(l_1,\hdots,l_r)=(1,2,\hdots,r)$ and $(u_1,\hdots,u_r)=(r+1,\hdots, 2r)$. Conversely, if
$2r>n$, then $\begin{pmatrix}x_1&x_2&\hdots &x_r\\ax_r&x_{r-1}&\hdots&x_1\end{pmatrix}$,
where $a\neq 1_G$, cannot lie in $P$, since we would need $l_{1\overline{\alpha}}=l_r<u_1$; by the pigeon-hole principle, this is not possible.

Consider now the case where $|G|=1$. If $2r\leq n+1$, then given any
$$\alpha=\begin{pmatrix}x_1&x_2&\hdots&x_r\\x_{1\overline{\alpha}}&x_{2\overline{\alpha}}&\hdots &x_{r\overline{\alpha}}\end{pmatrix}$$
in $ H$, let
$1\overline{\alpha}=t$ and choose
\[(l_1,\hdots,l_r)=(1,\hdots,r)\mbox{ and }(u_1,\hdots,u_r)=(t,r+1,\hdots, 2r-1).\]
It follows from the discussion preceding the lemma that $\alpha\in P$. Conversely, if
$2r>n+1$, then $\begin{pmatrix}x_1&x_2&\hdots &x_r\\x_r&x_{r-1}&\hdots&x_1\end{pmatrix}$
cannot lie in $P$, since now we would require $l_{1\overline{\alpha}}=l_r\leq u_1$.

\end{proof}

We are now in a position to outline the proof of our main theorem, Theorem~\ref{thm:n-2}, which states that
$\overline{H}$ is isomorphic to $H$, and hence to  $G\wr \mathcal{S}_r$.

We  first claim  that for any $i,j\in I$ and $\lambda,\mu\in \Lambda$, if $\mathbf{p}_{\lambda i}=\mathbf{p}_{\mu j}$,
then $f_{i,\lambda}=f_{j,\mu}$. We verify our claim via a series of steps. We first deal with the case where
$\mathbf{p}_{\lambda i}=\e$ and here show that $f_{i,\lambda}$ (and $f_{j,\mu}$) is the identity of $\overline{H}$
(Lemma~\ref{lem:idgens}). Next, we verify the claim in the case where $\mu=\lambda$ (Lemma~\ref{lem:easyform}) or $i=j$
(Lemma~\ref{lem:samecolumn}). We then show that for $r\leq \frac{n}{2}-1$, this is  sufficient (via finite induction)
to prove the claim holds in general (Lemma~\ref{lem:2r+1}). However, a counterexample shows that for larger $r$ this
strategy will fail.

To overcome the above problem, we begin by showing that if $\mathbf{p}_{\lambda i}=\mathbf{p}_{\mu j}$ is what we call a
{\em simple form}, that is,
\[\left(\begin{array}{ccccccccccc}
x_{1} & \cdots & x_{k-1} & x_{k} & x_{k+1} & \cdots & x_{k+m-1}&x_{k+m}&x_{k+m+1}&\hdots &x_{r} \\
x_{1} & \cdots & x_{k-1} & x_{k+1} & x_{k+2} & \cdots & x_{k+m}&ax_k&x_{k+m+1}&\hdots &x_r\end{array}\right),\] for
some $k\geq 1, m\geq 0, a\in G$, then $f_{i,\lambda}=f_{j,\mu}$. We then introduce the notion of {\em rising point} and
verify by induction on the rising point, with the notion of simple form forming the basis of our induction, that our
claim holds.  As a consequence of our claim  we denote a generator $f_{i,\lambda}$ with $\mathbf{p}_{\lambda i}=\phi$
by $f_{\phi}$.

For $r\leq \frac{n}{2}$ it is easy to see that every element of $H$ occurs as {\em some} $\mathbf{p}_{\lambda i}$ and
for $r\leq \frac{n}{3}$ we have enough room for manoeuvre (the reader studying Sections~\ref{sec:nonidgens} and 6 will come to
an understanding of what this means) to show that $f_\phi f_\varphi=f_{\varphi\phi}$ and it is then easy to see that
$\overline{H}\cong H$ (Theorem \ref{thm:n/3}).

To deal with the general case of  $r\leq n-2$ we face two problems. One is that for $r>\frac{n}{2}$, not every element
of $H$ occurs as some element of $P$ and secondly, we need more sophisticated techniques to show that the
multiplication in $\overline{H}$ behaves as we would like. To this end we show that $\overline{H}$ is generated by a restricted set of
elements $f_{i,\lambda}$, such that the corresponding $\mathbf{p}_{\lambda i}$ form a standard set of generators of $H$
(regarded as a wreath product). We then check that the corresponding identities to determine $G\wr S_r$ are satisfied
by these generators, and it is then a short step to obtain our goal, namely, that $\overline{H}\cong H$ (Theorem \ref{thm:n-2}). We note,
however, that even at this stage more care is required than, for example, in the corresponding situation for
$\mathcal{T}_n$ \cite{gray:2012a} or $\mathcal{PT}_n$ \cite{dolinka:2013}, since we cannot assume that $G$ is finite.
Indeed our particular choice of Schreier system will be seen to be a useful tool.

\section{Identity generators}\label{sec:identitygens}

As stated at the end of  Section~\ref{sec:pres}, our first step is to show that  if $(i,\lambda)\in K$ and
$\mathbf{p}_{\lambda i}=\e$, then $f_{i,\lambda}=1_{\overline{H}}$. Note that whenever we write $f_{i,\lambda}=1_{\overline{H}}$ we mean that this relation can be deduced from the relations in the presentation $\langle F: \Sigma\rangle.$ To this end we make use of our particular choice of
Schreier system and function $\omega$. The proof is by induction on $\lambda\in \Lambda$, where we recall that $\Lambda
$ is ordered lexicographically.

\begin{lem}\label{lem:idgens} For any $(i,\lambda)\in K$ with $\mathbf{p}_{\lambda i}=\e$, we have
$f_{i,\lambda}=1_{\overline{H}}$.\end{lem}
\begin{proof} After Lemma 3.7 we have already noted that $\bold{p}_{\omega(i),i}=\bold{q}_{\omega(i)}\bold{r}_i=\e$ for all $i\in I.$  If $\mathbf{p}_{(1,2,\hdots ,r)i}=\e$, that is, $\mathbf{q}_{(1, 2,\cdots,r)}\mathbf{r}_{i}=\e$, then
by definition of  $\mathbf{q}_{(1, 2,\cdots,r)}$ we have $x_1\mathbf{r}_i=x_1,\cdots, x_r\mathbf{r}_i=x_r$. Hence
$\mathbf{r}_{i}$ lies in district $(1,2,\cdots,r)$, so that $\omega(i)=(1,2,\cdots,r)$. Condition (R2) of the
presentation  $\mathcal{P}$ now gives that $f_{i,(1,2,\cdots,r)}=f_{i,\omega(i)}=1_{\overline{H}}$.

Suppose now that $\mathbf{p}_{(u_1,u_2,\hdots,u_r)i}=\e$ where $(1,2,\hdots, r)<(u_1,u_{2},\cdots,u_{r})$. We make the
inductive assumption  that for any $(v_1,v_{2},\cdots,v_{r})<(u_1,u_{2},\cdots,u_{r})$, if
$\mathbf{p}_{(v_1,v_{2},\cdots,v_{r})l}=\e$, for any $l\in I$, then $f_{l,(v_1,v_2,\hdots,v_r)}=1_{\overline{H}}$.

With $u_0=0$, pick the  largest number, say $j$, such that $u_{j}-u_{j-1}>1$. By our choice of Schreier words, we have
$$\mathbf{h}_{(u_1,u_{2},\cdots, u_{r})}=\mathbf{h}_{(u_1,u_{2},\cdots, u_{j-1}, u_{j}-1, u_{j+1},\cdots, u_{r})}\alpha_{(u_1,u_{2},\cdots, u_{r})},$$ where $\alpha_{(u_1,u_{2},\cdots, u_{r})}=\e_{l (u_1,u_2,\hdots, u_r)}$.


By definition,
\begin{align*}
\mathbf{r}_{l}=\left(\begin{array}{ccccccccccc}
x_{1} & \hdots &x_{u_1}&x_{u_1+1} & \cdots & x_{u_{2}}  & \cdots  & x_{u_{r-1}+1} & \cdots  \\
x_{1} & \hdots&x_1&x_{2} & \cdots & x_{2}   & \cdots  & x_{r} & \cdots
\end{array}\right. & \\
\left.\begin{array}{ccccc}
\cdots & x_{u_{r}}&x_{u_r+1}& \cdots & x_{n} \\
\cdots & x_{r} &x_r& \cdots & x_{r}
\end{array}\right).&
\end{align*}
By choice of $j$ we have $u_{j-1}<u_{j}-1<u_{j}$ so that $x_{u_{j}-1}\mathbf{r}_{l}=x_{j}$,
giving \[\mathbf{p}_{(u_1,u_{2},\cdots, u_{j-1}, u_{j}-1, u_{j+1},\cdots, u_{r})l}=\e.\] Since $$(u_1,u_{2},\cdots,
u_{j-1}, u_{j}-1, u_{j+1},\cdots, u_{r})<(u_1,u_{2},\cdots, u_{j-1},u_{j},u_{j+1},\cdots,u_{r}),$$
 we call upon our inductive hypothesis to obtain $f_{l,(u_1,u_{2},\cdots, u_{j-1}, u_{j}-1, u_{j+1},\cdots, u_{r})}=1_{\overline{H}}$.  On the other hand, we have $f_{l,(u_1,u_{2},\cdots,u_{r})}=f_{l,(u_1,u_{2},\cdots, u_{j-1}, u_{j}-1, u_{j+1},\cdots, u_{r})}$ by
 (R1), and so we conclude that $f_{l,(u_1,u_{2},\cdots,u_{r})}=1_{\overline{H}}$.

Suppose that $\mathbf{r}_{i}$ lies in district $(l_1,l_{2},\cdots,l_{r})$. Since
$\mathbf{q}_{(u_1,u_{2},\cdots,u_{r})}\mathbf{r}_{i}=\e$, we have $x_{u_{k}}\mathbf{r}_{i}=x_{k}$, so that $l_{k}\leq
u_{k}$ by the definition of districts, for all $k\in [1,r]$. If $l_k=u_k$ for all $k\in [1,r]$, then $f_{i,(u_1,\hdots,
u_r)}= f_{i,\omega(i)}=1_{\overline{H}}$  by $\mathcal{P}$. Otherwise, we let $m$ be smallest such that   $l_{m}<u_{m}$
and so (putting $u_0=l_0=0$)  we have $u_{m-1}=l_{m-1}<l_{m}<u_{m}$. Clearly $(u_1,u_{2},\cdots,
u_{m-1},l_{m},u_{m+1},\cdots ,u_{r})\in\Lambda$ and as  $u_{m-1}<l_{m}<u_{m}$, we have $x_{l_{m}}\mathbf{r}_{l}=x_{m}$
by the definition of $\mathbf{r}_{l}$. We  thus have  the matrix equality
$$\left(\begin{array}{cc}
\mathbf{q}_{(u_1,u_{2},\cdots,u_{r})}\mathbf{r}_{l} & \mathbf{q}_{(u_1,u_{2},\cdots,u_{r})}\mathbf{r}_{i}\\
\mathbf{q}_{(u_1,u_{2},\cdots, u_{m-1}, l_{m}, u_{m+1},\cdots,u_{r})}\mathbf{r}_{l} & \mathbf{q}_{(u_1,u_{2},\cdots,
u_{m-1}, l_{m}, u_{m+1},\cdots,u_{r})}\mathbf{r}_{i}\end{array}\right)=\left(\begin{array}{cc}
\e & \e\\
\e & \e \end{array}\right).$$

Remember that we have already proven $f_{l,(u_1,u_{2},\cdots,u_{r})}=1_{\overline{H}}$. Furthermore, as $l_{m}<u_{m}$
by assumption, $$(u_1,u_{2},\cdots, u_{m-1}, l_{m}, u_{m+1},\cdots,u_{r})<(u_1,u_{2},\cdots,
u_{m-1},u_{m},u_{m+1},\cdots,u_{r}),$$  so that induction gives that  $$f_{i, (u_1,u_{2},\cdots, u_{m-1}, l_{m},
u_{m+1},\cdots,u_{r})}=f_{l,(u_1,u_{2},\cdots, u_{m-1}, l_{m}, u_{m+1},\cdots,u_{r})}=1_{\overline{H}}.$$ From (R3) we
deduce that $f_{i,(u_1,u_{2},\cdots,u_{r})}=1_{\overline{H}}$ and the proof is completed.
\end{proof}

\section{Generators corresponding to  the same columns or rows, and connectivity}\label{sec:nonidgens}

The first aim of this section is to show that if $\mathbf{p}_{\lambda i}=\mathbf{p}_{\mu j}\neq 0$ where $\lambda =\mu$
or $i=j$, then $f_{i,\lambda}=f_{j,\mu}$. We begin with the more straightforward case, where $i=j$.

\begin{lem}\label{lem:easyform} If $\mathbf{p}_{\lambda i}=\mathbf{p}_{\mu i}$, then
$f_{i,\lambda}=f_{i,\mu}$.
\end{lem}
\begin{proof} Let
$\lambda= (u_{1},\cdots,u_{r})$ and $\mu =(v_1,\hdots, v_r)$. By hypothesis we have that
$\mathbf{q}_{(u_{1},\cdots,u_{r})}\mathbf{r}_{i}=\mathbf{q}_{(v_1,\cdots, v_r)}\mathbf{r}_{i}=\psi\in H$. By definition
of the $\mathbf{q}_{\lambda}$s we have $x_{u_j}\mathbf{r}_i=x_{v_j}\mathbf{r}_i$ for $1\leq j\leq r$, and as $\rank
\mathbf{r}_i=r$ it follows that $u_j,v_j\in B^{\mathbf{r}_i}_{j'}$ where $j\mapsto j'$ is a bijection of $[1,r]$. We
now define $\alpha\in \en F_n(G)$ by setting $x_{u_{j}}\alpha=x_{j}=x_{v_j}\alpha$ for all $j\in [1,r]$ and
$x_{p}\alpha=x_{1}$ for all $p\in [1,n]\setminus \{ u_1,\cdots, u_r,v_1,\cdots v_r\}$.

Clearly $\alpha\in D_r$, indeed $\alpha\in L_1$. Since $w^{\alpha}_m=1_G$ for all $m\in [1,n]$ and min$\{
u_j,v_j\}<$min$\{ u_k, v_k\}$ for $1\leq j<k\leq r$, we certainly have that $\alpha=\mathbf{r}_l$ for some $l\in I$. By
our choice of $\mathbf{r}_l$ we have the matrix equality

 $$\left(\begin{array}{cc}
\mathbf{q}_{(u_{1},\cdots, u_{r})}\mathbf{r}_{i} & \mathbf{q}_{(u_{1},\cdots, u_{r})}\mathbf{r}_{l}\\
\mathbf{q}_{(v_1,\cdots, v_r)}\mathbf{r}_{i} & \mathbf{q}_{(v_1,\cdots, v_r)}\mathbf{r}_{l}
\end{array}\right)=\left(\begin{array}{cc}
\psi & \e\\
\psi & \e \end{array}\right).$$ Using Lemma~\ref{lem:idgens} and (R3) of the presentation $\mathcal{P}$, we obtain
$$f_{i,(u_{1},\cdots,u_{r})}=f_{i,(v_1,\cdots,v_r)}$$ as required.
\end{proof}

We need more effort for the case $\mathbf{p}_{\lambda i}=\mathbf{p}_{\lambda j}$. For this purpose we introduce the
following notions of `bad' and `good' elements.

For any $i,j\in I$, suppose that $\mathbf{r}_{i}$ and $\mathbf{r}_{j}$ lie in districts $(1,k_{2},\cdots, k_{r})$ and
$(1,l_{2},\cdots,l_{r})$, respectively. We call $u\in [1,n]$ a mutually {\it  bad} element of $\mathbf{r}_{i}$ with
respect to $\mathbf{r}_{j}$, if there exist $m,s\in [1,r]$ such that $u=k_{m}=l_{s}$, but $m\neq s$; all other elements
are said to be mutually {\em good} with respect to $\mathbf{r}_i$ and $\mathbf{r}_j$. We call $u$  a bad element of
$\mathbf{r}_{i}$ with respect to $\mathbf{r}_{j}$ because, from the definition of districts,  $\mathbf{r}_{i}$  maps
$x_{k_{m}}$ to $x_{m}$, and similarly, $\mathbf{r}_{j}$ maps $x_{l_{s}}$ to $x_{s}$. Hence, if $u=k_{m}=l_{s}$ is bad,
then it is impossible for us to find some $\mathbf{r}_{t}$ to make both $\mathbf{r}_{i}$ and $\mathbf{r}_{j}$  `happy'
in the point $x_u$, that is, for $\mathbf{r}_{t}$ (or, indeed, any other element of $\en F_n(G)$) to agree with both
$\mathbf{r}_{i}$ and $\mathbf{r}_{j}$ on $x_u$.

 Notice that if $m$ is the minimum subscript such that $u=k_{m}$ is a bad element of $\mathbf{r}_{i}$ with respect to $\mathbf{r}_{j}$ and $k_{m}=l_{s}$, then $s$ is also the minimum subscript such that $l_{s}$ is a  bad element of $\mathbf{r}_{j}$ with respect to $\mathbf{r}_{i}$. For, if $l_{s'}<l_{s}$ is a bad element of $\mathbf{r}_{j}$ with respect to $\mathbf{r}_{i}$, then by definition we have some $k_{m'}$ such that $l_{s'}=k_{m'}$ where $s'\neq m'$. By the minimality of  $m$, we have $m'>m$ and so $l_{s'}=k_{m'}>k_{m}=l_{s}$, a contradiction.
 We also remark that since $l_1=k_1=1$, the maximum possible number of bad elements  is $r-1$.

Let us run a simple example. Let $n=7$ and $r=4$, and suppose $\mathbf{r}_{i}$ lies in district $(1,3,4,6)$ and
$\mathbf{r}_{j}$ lies in district $(1,4,6,7)$. By definition,  $x_{4}\mathbf{r}_i=x_{3}$ and $x_{6}\mathbf{r}_i=x_{4}$,
while $x_{4}\mathbf{r}_j=x_{2}$ and $x_{6}\mathbf{r}_j=x_{3}$. Therefore,  $\mathbf{r}_{i}$ and $\mathbf{r}_{j}$ differ
on $x_{4}$ and $x_{6}$, so  that we say $4$ and $6$ are bad elements of $\mathbf{r}_{i}$ with respect to
$\mathbf{r}_{j}$.

\begin{lem}\label{lem:goodset}
For any $i,j\in I$, suppose that $\mathbf{r}_{i}$ and $\mathbf{r}_{j}$ lie in districts $(1,k_{2},\cdots,k_{r})$ and
$(1,l_{2}, \cdots,l_{r})$, respectively. Let
$\mathbf{q}_{(u_{1},\cdots,u_{r})}\mathbf{r}_{i}=\mathbf{q}_{(u_{1},\cdots,u_{r})}\mathbf{r}_{j}=\psi\in H$. Suppose
$\{1,l_{2},\cdots,l_{s}\}$ is a  set of good elements of $\mathbf{r}_{i}$ with respect to $\mathbf{r}_{j}$ such that
$1<l_{2}<\cdots<l_{s}<k_{s+1}<\cdots<k_{r}$. Then there exists $ p\in I$ such that
$\mathbf{r}_p$ lies in district $(1,l_{2},\cdots,l_{s}, k_{s+1},\cdots, k_{r})$,
$\mathbf{q}_{(u_{1},\cdots,u_{r})}\mathbf{r}_{p}=\psi$ and $f_{p,(u_{1},\cdots,u_{r})}=f_{i,(u_{1},\cdots,u_{r})}$.
Further, if $s=r$ then we can take $p=j$.
\end{lem}
\begin{proof} We begin by defining $\alpha\in D_r$, starting by setting $x_{k_m}\alpha=x_m$, $m\in [1,r]$. Now for
$m\in[1,s]$ we put $x_{l_m}\alpha=x_m$. Notice that for $1\leq m\leq s$, if $k_{m'}=l_m$ for $m'\in[1,r]$, then by the
goodness of $\{1,l_{2},\cdots,l_{s}\}$ we have that $m'=m$. We now set $x_{u_m}\alpha=x_{u_m}\mathbf{r}_i$ for $m\in
[1,r]$. Again, we need to check we are not violating well-definedness. Clearly we need only check the case where
$u_m=l_{m'}$ for some $m'\in[1,s]$, since here we have already defined $x_{l_{m'}}\alpha=x_{m'}$. We now use the fact
that by our hypothesis, $x_{u_m}\mathbf{r}_i= x_{u_m}\mathbf{r}_j$ for all $m\in [1,r]$, so that
$x_{u_m}\mathbf{r}_i=x_{u_m}\mathbf{r}_j=x_{l_{m'}}\mathbf{r}_j=x_{m'}$. Finally, we set $x_m\alpha=x_1$, for all $m\in
[1,n]\setminus \{ 1,l_1,\cdots, l_s, k_2,\cdots, k_r,u_1,\cdots, u_r\}$.

We claim that $\alpha=\mathbf{r}_t$ for some $t\in I$. First, it is clear from the definition that $\alpha\in D_r$,
indeed, $\alpha\in L_1$. We  also have that for $1\leq m\leq s$, $x_{l_m}\alpha=x_{k_m}\alpha=x_m$ and also for
$s+1\leq m\leq r$, $x_{k_m}\alpha=x_m$. We claim that for $m\in [1,s]$ we have $l^{\alpha}_m=v_m$ where $v_m=\mbox{ min
}\{k_m,l_m\}$ and for $m\in [s+1,n]$ we have $l^{\alpha}_m=k_m$. It is clear that $1=l^{\alpha}_1$. Suppose that for
$m\in [2,r]$ we have $x_{u_q}\alpha=ax_m$. By definition, $x_{u_q}\mathbf{r}_i=ax_m=x_{u_q}\mathbf{r}_j$, so that
$k_m,l_m\leq u_q$ and our claim holds. It is now clear that $\alpha=\mathbf{r}_t$ for some $t\in I$ and lies in
district $(v_1,\cdots, v_s,k_{s+1},\cdots, k_r)$.

Having constructed $\mathbf{r}_t$, it is immediate that $$\left(\begin{array}{cc}
\mathbf{q}_{(1,k_{2},\cdots,k_{r})}\mathbf{r}_{i} & \mathbf{q}_{(1,k_{2},\cdots,k_{r})}\mathbf{r}_{t}\\
\mathbf{q}_{(u_{1},u_{2},\cdots,u_{r})}\mathbf{r}_{i} & \mathbf{q}_{(u_{1},u_{2},\cdots,u_{r})}\mathbf{r}_{t}
\end{array}\right)=\left(\begin{array}{cc}
\epsilon & \epsilon\\
\psi & \psi \end{array}\right),$$ so that in view of Corollary~\ref{lem:ssindr} and (R3) we deduce that
$$f_{i,(u_{1},u_{2}\cdots,u_{r})}=f_{t,(u_{1},u_{2},\cdots,u_{r})}.$$

Notice now that if $s=r$ then
$$\left(\begin{array}{cc}
\mathbf{q}_{(1,l_{2},\cdots,l_{r})}\mathbf{r}_{t} & \mathbf{q}_{(1,l_{2},\cdots,l_{r})}\mathbf{r}_{j}\\
\mathbf{q}_{(u_{1},u_{2},\cdots,u_{r})}\mathbf{r}_{t} & \mathbf{q}_{(u_{1},u_{2},\cdots,u_{r})}\mathbf{r}_{j}
\end{array}\right)=\left(\begin{array}{cc}
\epsilon & \epsilon\\
\psi & \psi \end{array}\right),$$ which leads to $f_{t,(u_{1},u_{2},\cdots,u_{r})}=f_{j,(u_{1},u_{2},\cdots,u_{r})}$, and so that $f_{j,(u_{1},\cdots,u_{r})}=f_{i,(u_{1},\cdots,u_{r})}$ as required.

Without the assumption that $s=r$, we now define $\mathbf{r}_p$ in a similar, but slightly more straightforward way, to
$\mathbf{r}_t$. Namely, we first define $\beta\in \en F_n(G)$ by putting $x_{l_m}\beta=x_m$ for $m\in [1,s],
x_{k_m}\beta=x_m$ for $m\in [s+1,r]$, $x_{u_m}\beta=x_{u_m}\mathbf{r}_i$ for $m\in [1,r]$ and $x_m\beta=x_1$ for $m\in
[1,n]\setminus \{ 1,l_1,\cdots, l_s, k_{s+1},\cdots, k_r,u_1,\cdots, u_r\}$. It is easy to check that
$\beta=\mathbf{r}_p$ where $\mathbf{r}_p$ lies in district $(1,l_2,\cdots,l_s,k_{s+1},\cdots, k_r)$. Moreover, we have
$$\left(\begin{array}{cc}
\mathbf{q}_{(1,l_{2},\cdots,l_{s},k_{s+1},\cdots,k_{r})}\mathbf{r}_{t} & \mathbf{q}_{(1,l_{2},\cdots,l_{s},k_{s+1},\cdots,k_{r})}\mathbf{r}_{p}\\
\mathbf{q}_{(u_{1},u_{2},\cdots,u_{r})}\mathbf{r}_{t} & \mathbf{q}_{(u_{1},u_{2},\cdots,u_{r})}\mathbf{r}_{p}
\end{array}\right)=\left(\begin{array}{cc}
\epsilon & \epsilon\\
\psi & \psi \end{array}\right),$$ which leads to $f_{t,(u_{1},u_{2},\cdots,u_{r})}=f_{p,(u_{1},u_{2},\cdots,u_{r})}$,
and so to $f_{p,(u_{1},\cdots,u_{r})}=f_{i,(u_{1},\cdots,u_{r})}$ as required.\end{proof}

\begin{lem}\label{lem:samecolumn}
If  $\mathbf{p}_{\lambda i}=\mathbf{p}_{\lambda j}$, then $f_{i,\lambda}=f_{j,\lambda}$.
\end{lem}
\begin{proof}
Suppose that $\mathbf{r}_{i}$ and $\mathbf{r}_{j}$ lie in districts $(1,k_{2},\cdots,k_{r})$ and $(1,l_{2},
\cdots,l_{r})$, respectively. Let $\lambda=(u_1,\hdots, u_r)$ so that
 $\mathbf{q}_{(u_{1},\cdots,u_{r})}\mathbf{r}_{i}=\mathbf{q}_{(u_{1},\cdots,u_{r})}\mathbf{r}_{j}=\psi\in H$ say.  We proceed by induction on the number of mutually bad elements. If this is 0, then the result holds by Lemma~\ref{lem:goodset}.
We make the inductive assumption that if $\mathbf{p}_{\lambda l}=\mathbf{p}_{\lambda t}$ and
$\mathbf{r}_l,\mathbf{r}_t$ have $k-1$ bad elements, where $0< k\leq r-1$, then $f_{l,\lambda}=f_{t,\lambda}$.

Suppose now that $\mathbf{r}_{j}$ has $k$ bad elements with respect to $\mathbf{r}_{i}$.  Let $s$ be the smallest
subscript such that $l_{s}$ is bad element of $\mathbf{r}_{j}$ with respect to $\mathbf{r}_{i}$. Then there exists some
$m$ such that $l_{s}=k_{m}$. Note, $m$ is also the smallest subscript such that $k_{m}$ is bad, as  we explained
before. Certainly $s,m>1$; without loss of generality, assume $s>m$. Then $1=l_1,l_{2}, \cdots, l_{s-1}$ are all good
elements and $1<l_{2}<\cdots<l_{s-1}<k_{s}<\cdots<k_{r}$. By Lemma \ref{lem:goodset}, there exists $p\in I$ such that
$\mathbf{r}_p$ lies in district $(1,l_{2},\cdots,l_{s-1}, k_{s},\cdots, k_{r})$,
$\mathbf{q}_{(u_{1},\cdots,u_{r})}\mathbf{r}_{p}=\psi$ and $f_{p,(u_{1},\cdots,u_{r})}=f_{i,(u_{1},\cdots,u_{r})}$.

We consider the sets $B$ and $C$ of  mutually bad elements of $\mathbf{r}_j$ and $\mathbf{r}_p$, and of $\mathbf{r}_j$
and $\mathbf{r}_i$, respectively. Clearly $B\subseteq \{ l_s,l_{s+1},\cdots,l_r\}$. We have $l_s=k_m<k_s$, so that
$l_s\notin B$.   On the other hand if $l_q\in B$ where $s+1\leq  q\leq r$, then we must have $l_q=k_{q'}$ for some
$q'\geq s$ with $q'\neq q$, so that $l_q\in C$. Thus $|B|<|C|=k$. Our inductive hypothesis now gives that
$f_{p,(u_{1},\cdots,u_{r})}=f_{j,(u_{1},\cdots,u_{r})}$ and we deduce that
$f_{i,(u_{1},\cdots,u_{r})}=f_{j,(u_{1},\cdots,u_{r})}$ as required.\end{proof}

\begin{defn}\label{defn:connected} Let $i,j\in I$ and $\lambda,\mu\in \Lambda$
such that $\mathbf{p}_{\lambda i}=\mathbf{p}_{\mu j}$. We say that $(i,\lambda),(j,\mu)$ are {\em connected} if there
exist
\[i=i_0,i_1,\hdots, i_m=j\in I\mbox{ and }\lambda=\lambda_0,\lambda_1,\hdots, \lambda_m=\mu\in \Lambda\]
such that for $0\leq k< m$ we have $\mathbf{p}_{\lambda_k i_k}=\mathbf{p}_{\lambda_k, i_{k+1}}=\mathbf{p}_{\lambda_{k+1}i_{k+1}}$.
\end{defn}

The following picture illustrates that $(i,\lambda)=(i_0,\lambda_0)$ is connected to $(j,\mu)=(i_m,\lambda_m)$:

\begin{center}
\begin{tikzpicture}[scale=1.0]
\node (a) at (-6,3) {$\mathbf{p}_{\lambda_0 i_0}$}; \node (b) at (-4,3) {$\mathbf{p}_{\lambda_0 i_1}$}; \node (c) at (-4,2)
{$\mathbf{p}_{\lambda_1 i_1}$}; \node (d) at (-2,2) {$\mathbf{p}_{\lambda_1 i_2}$}; \node (e) at (-2,1) {}; \node (f)
at (0,1) {}; \node (g) at (0,0) {$\mathbf{p}_{\lambda_{m-1} i_{m-1}}$}; \node (h) at (3,0) {$\mathbf{p}_{\lambda_{m-1}
i_{m}}$}; \node (k) at (3,-1) {$\mathbf{p}_{\lambda_{m} i_{m}}$}; \path[-,font=\scriptsize,>=angle 60] (a) edge
node[above]{} (b) (b) edge node[right]{} (c) (c) edge node[right]{} (d); \path[-,dashed,font=\scriptsize,>=angle 60] (d) edge node[above]{} (e)
(e) edge node[right]{} (f) (f) edge node[above]{} (g); \path[-,font=\scriptsize,>=angle 60] (g) edge node[right]{} (h) (h) edge node[right]{} (k);
\end{tikzpicture}
\end{center}

Lemmas~\ref{lem:easyform} and ~\ref{lem:samecolumn} now yield:

\begin{coro}\label{coro:connected} Let $i,j\in I$ and $\lambda,\mu\in \Lambda$ be such that  $\mathbf{p}_{\lambda i}=\mathbf{p}_{\mu j}$
where  $(i,\lambda),(j,\mu)$ are
{ connected}. Then $f_{i,\lambda}=f_{j,\mu}$.
\end{coro}

\section{The result for restricted $r$}\label{sec:restricted}

We are now in a position to finish the proof of our first main result, Theorem~\ref{thm:n/3}, in a relatively straightforward way. Of course, in view of Theorem~\ref{thm:n-2}, it is  not strictly necessary to provide such a proof here. However, the  techniques used  will be useful in the remainder of the paper.

Let $\alpha=\mathbf{p}_{\lambda i}\in P$ and suppose that $\lambda=(u_1,\cdots,u_r)$ and $\mathbf{r}_i$ lies in
district $(l_1,\cdots,l_r).$ Define $U(\lambda,i)=\{l_1, \cdots, l_r, u_1,\cdots, u_r\}$ and
$S(\lambda,i)=[1,n]\setminus U(\lambda, i)$.

\bigskip\noindent{\bf Step D: moving $l$s down:} Suppose that $l_j<t<l_{j+1}$ and
$t\in S(\lambda,i)$.  Define $\mathbf{r}_k$ by
\[ x_t\mathbf{r}_k=x_{j+1}\mbox{ and }x_s\mathbf{r}_j=x_s\mathbf{r}_i\mbox{ for }s\neq t.\]
It is easy to see that $\mathbf{r}_k\in\Theta$, $\mathbf{p}_{\lambda i}=\mathbf{p}_{\lambda k}$ and $\mathbf{r}_k$ lies
in district $$(l_1,\hdots, l_j,t,l_{j+2},\hdots ,l_r).$$ Clearly, $(i, \lambda)$ is connected to $(k, \lambda).$

\bigskip\noindent
{\bf Step U: moving $u$s up:} Suppose that $u_j<t<u_{j+1}$ or $u_r<t$, where $t\in S(\lambda,i)$. Define $\mathbf{r}_m$
by
\[x_t\mathbf{r}_m=x_{u_{j}}\mathbf{r}_i\mbox{ and }x_s\mathbf{r}_m=x_s\mathbf{r}_i\mbox{ for }s\neq t. \]
It is easy to see that $\mathbf{r}_m\in\Theta$, $\mathbf{p}_{\lambda i}=\mathbf{p}_{\lambda m}$ and $\mathbf{r}_m$ lies
in district $(l_1,l_2,\hdots ,l_r)$. Let $\mu=(u_1,\hdots, u_{j-1},t,u_{j+1},\hdots, u_r)$. Clearly,
$\mathbf{p}_{\lambda m}=\mathbf{p}_{\mu m}$ so that $(i,\lambda)$ is connected to $(m,\mu)$.

\bigskip\noindent
{\bf Step U$'$: moving $u$s down:}  Suppose that $t<u_{j+1}$ and
 $[t,u_{j+1})\subseteq S(\lambda,i)$.
Define $\mathbf{r}_l$ by
\[x_t\mathbf{r}_l=x_{u_{j+1}}\mathbf{r}_i\mbox{ and }x_s\mathbf{r}_m=x_s\mathbf{r}_i\mbox{ for }s\neq t. \]
It is easy to see that $\mathbf{r}_l\in\Theta$, $\mathbf{p}_{\lambda i}=\mathbf{p}_{\lambda l}$. Further,
 $\mathbf{r}_l$ lies in district
$(l_1,l_2,\hdots ,l_r)$ unless $u_{j+1}=l_{(j+1)\overline{\alpha}}$, in which case $l_{(j+1)\overline{\alpha}}$ is replaced by
 $t$.  Let $$\mu=(u_1,\hdots, u_{j},t,u_{j+2},\hdots, u_r);$$ clearly, $\mathbf{p}_{\lambda l}=\mathbf{p}_{\mu l}$,
 so that $(i,\lambda)$ is connected to $(l,\mu)$.


\begin{lem}\label{lem:2r+1}
Suppose that $n\geq 2r+1$. Let $\lambda=(u_1,\cdots,u_r)\in \Lambda$, and $i\in I$ with $\mathbf{p}_{\lambda i}\in H$.
Then $(i,\lambda)$ is connected to $(j,\mu)$ for some $j\in I$ and  $\mu=(n-r+1, \cdots,n)$. Consequently, if
$\mathbf{p}_{\lambda i}=\mathbf{p}_{\nu k}$ for any $i,k\in I$ and $\lambda,\nu\in \Lambda$, then $
f_{i,\lambda}=f_{k,\nu}$.
\end{lem}
\begin{proof}  Suppose that $\bold{r}_i$ lies in district $(l_1, \cdots, l_r).$ For the purposes of this proof, let $$W(\lambda,i)=\sum^{r}_{k=1}(u_k-l_k);$$ clearly
$W(\lambda,i)$ takes greatest value $T$ where $$(l_1,\hdots,l_r)=(1,\hdots,r)\mbox{~and~}
(u_1,\hdots, u_r)=(n-r+1,\hdots, n).$$ Of course, here $W(\lambda,i)$ can be a negative integer, however, it has a minimal value that it can attain, i.e. is bounded below. We verify our claim by finite induction, with starting point $T$, under the reverse of the usual ordering on $\mathbb{Z}$. We have remarked that our result holds if $W(\lambda,i)=T$.

Suppose now that $W(\lambda,i)<T$ and the result is true for all pairs $(\nu,l)$ where
$W(\lambda,i)<W(\nu,l)\leq T$.

If $u_r<n$, then as certainly $l_{r}\leq u_{r\overline{\mathbf{p}_{\lambda i}}^{-1}}\leq u_r$, we can apply Step U to
show that $(i,\lambda)$ is connected to $(l, \nu)$ where $\nu=(u_1,\hdots,u_{r-1},u_r+1)$ and $\mathbf{r}_l$ lies in
district $(l_1,\hdots,l_r)$. Clearly $W(\lambda,i)<W(\nu,l)$.

Suppose that $u_r=n$. We know that $l_1=1$, and by our hypothesis that $2r+1\leq n$, certainly
$S(\lambda,i)\neq\varnothing$. If there exists $t\in S(\lambda,i)$ with $t<l_w$ for some $w\in [1,r]$, then choosing
$k$ with $l_k<t<l_{k+1}$, we have by Step D that $(i,\lambda)$ is connected to $(l, \lambda)$, where $\mathbf{r}_l$ lies
in district $(l_1,\hdots, l_k,t,l_{k+2},\hdots,l_r)$; clearly then $W(\lambda,i)<W(\lambda,l)$. On the other hand, if
there exists $t\in S(\lambda,i)$ with $u_w<t$ for some $w\in [1,r]$, then now choosing $k\in [1,r]$ with
$u_k<t<u_{k+1}$, we use Step U to show that $(i,\lambda)$ is connected to $(m,\nu)$ where
$\nu=(u_1,\hdots,u_{k-1},t,u_{k+1},\hdots, u_r)$, and $\mathbf{r}_m$ lies in district $(l_1,\hdots,l_r)$. Again,
$W(\lambda,i)<W(\nu,m)$.

The only other possibility is that $S(\lambda,i)\subseteq (l_r,u_1)$, in which case, $W(\lambda,i)=T$, a contradiction.
\end{proof}

In view of Lemma~\ref{lem:2r+1} and Lemma ~\ref{lem:all}, we may define, for $r\leq \frac{n-1}{2}$ and $\phi\in H$, an element $f_{\phi}\in
\overline{H}$, where $f_{\phi}=f_{i, \lambda}$ for some (any) $(i,\lambda)\in K$ with $\mathbf{p}_{\lambda i}=\phi$.

\begin{lem} \label{lem:morphism} Let $r\leq n/3$. Then
for any $\phi,\theta\in H$, we have $f_{\phi\theta}=f_{\theta}f_{\phi}$ and $f_{\phi^{-1}}=f_{\phi}^{-1}$.
\end{lem}
\begin{proof}  Since $n\geq 3$ and $r\leq n/3$ we deduce that $2r+1\leq n$.
 Define $\mathbf{r}_{i}$ by
 \[x_{j}\mathbf{r}_{i}=x_{j},j\in [1,r];\, x_{j}\mathbf{r}_{i}=x_{j-r}\phi\theta,j\in [r+1, 2r];\, x_{j}\mathbf{r}_{i}=x_{j-2r}\theta, j\in [2r+1,3r]\]
 and
 \[  x_{j}\mathbf{r}_i=x_{1},j\in[3r+1,n].\] Clearly, $\mathbf{r}_{i}\in\Theta$
 and $\mathbf{r}_i$ lies in district $(1,\cdots,r)$. Next we define $\mathbf{r}_{l}$ by
 \[ x_{j}\mathbf{r}_{l}=x_{j}, j\in [1,r];\, x_{j}\mathbf{r}_{l}=x_{j-r}\phi, j\in [r+1, 2r];\, x_{j}\mathbf{r}_{l}=x_{j-2r}, j\in [2r+1,3r];\]
 and
 \[ x_{j}\mathbf{r}_{l}=x_{1}, j\in [3r+1,n].\] Again, $\mathbf{r}_{l}$ is  well defined and lies in district $(1,\cdots,r)$. By considering
 the submatrix
$$\left(\begin{array}{cc}
\mathbf{q}_{(r+1,\cdots, 2r)}\mathbf{r}_{l} & \mathbf{q}_{(r+1,\cdots, 2r)}\mathbf{r}_{i}\\
\mathbf{q}_{(2r+1,\cdots,3r)}\mathbf{r}_{l} & \mathbf{q}_{(2r+1,\cdots,3r)}\mathbf{r}_{i}
\end{array}\right)=\left(\begin{array}{cc}
\phi & \phi\theta\\
\epsilon & \theta\end{array}\right),$$
of $P$, Corollary~\ref{lem:ssindr} gives that $f_{i,(r+1,\cdots,2r)}=f_{i,(2r+1,\cdots, 3r)}f_{l,(r+1,\cdots,2r)}$, which in our new notation says  $f_{\phi\theta}=f_{\theta}f_{\phi}$, as required.

Finally, since $1_{\overline{H}}=f_{\epsilon}=f_{\phi\phi^{-1}}=f_{\phi^{-1}}f_{\phi}$, we have $f_{\phi^{-1}}=f_{\phi}^{-1}$.
\end{proof}

\begin{them}\label{thm:n/3} Let $r\leq n/3$. Then $\overline{H}$ is isomorphic to $H$ under
${\boldsymbol \psi}$, where $f_{\phi}{\boldsymbol \psi}=\phi^{-1}$.
\end{them}
\begin{proof} We have that
$\overline{H}=\{ f_{\phi}:\phi\in H\}$ by Lemma~\ref{lem:morphism} and
${\boldsymbol \psi}$ is well defined, by Lemma~\ref{lem:2r+1}. By Lemma~\ref{lem:all}, ${\boldsymbol \psi}$ is onto and it is a homomorphism by
Lemma~\ref{lem:morphism}. Now $f_{\phi}{\boldsymbol \psi}=\epsilon$ means that $\phi=\epsilon$, so that
$f_{\phi}=1_{\overline{H}}$ by Lemma~\ref{lem:idgens}. Consequently, ${\boldsymbol \psi}$ is an isomorphism as required.
\end{proof}

\section{Non-identity generators with simple form}\label{sec:simpleforms}

First we explain the motivation for this section. It follows from Section \ref{sec:restricted} that for any $r$ and $n$ with $n\geq 2r+1$, all entries in the sandwich matrix $P$ are connected. However, this connectivity  will fail for higher ranks. Hence, the aim here is to identify the connected entries in $P$ in the case of higher rank. It turns out that entries with {\it simple form} are always connected. For the reason given in the abstract, from now on we may assume that $1\leq r\leq n-2.$

We run an easy example to explain the lack of connectivity for $r\geq n/2$.

Let $n=4$, $r=2$, and \[\alpha=\left(\begin{array}{cc}
x_{1} & x_{2}  \\
ax_{1} & bx_{2}\end{array}\right),\] with $a,b\neq 1_G\in G.$ It is clear from Lemma 3.10 that there exists $i\in I$,
$\lambda\in \Lambda$ such that $\alpha=\mathbf{p}_{\lambda i}\in P$, in fact we can take
\[\mathbf{r}_i=\left(\begin{array}{cccc}
x_{1} & x_{2} &  x_{3} & x_{4} \\
x_{1} & x_{2} &  ax_{1} & bx_{2}\end{array}\right)\] and $\lambda=(3,4).$

How many copies of $\alpha$ occur in the sandwich matrix $P$? Suppose that $\alpha=\mathbf{p}_{\mu j}$ where
$\mathbf{r}_j$ lies in district $(l_1,l_2)$ and $\mu=(u_1, u_2).$ Since $\overline{\alpha}$ is the identity of
$\mathcal{S}_2$, and $a,b \neq 1_G,$ we must have $1=l_1<l_2,$ $u_1<u_2$, $l_1<u_1$, $l_2<u_2$ and
$\{l_1,l_2\}\cap\{u_1, u_2\}=\varnothing.$ Thus the only possibilities are $$(l_1,l_2)=(1,2), (u_1,
u_2)=(3,4)=\lambda$$ and $$(l_1,l_2)=(1,3), (u_1, u_2)=(2,4)=\mu.$$ In the first case, $\alpha=\mathbf{p}_{\lambda i}$
and in the second, $\alpha=\mathbf{p}_{\mu j}$ where $$\mathbf{r}_j=\left(\begin{array}{cccc}
x_{1} & x_{2} &  x_{3} & x_{4} \\
x_{1} & ax_{1} &  x_{2} & bx_{2}\end{array}\right).$$ Clearly then, $\mathbf{p}_{\lambda i}=\mathbf{p}_{\mu j}\in H$
but $(i, \lambda)$ is {\it not} connected to $(j, \mu).$

\medskip

We know from Lemma~\ref{lem:all}, that in case $r\geq n/2$, not every element of $H$ lies in $P$. However, we are
guaranteed that for $r\leq n-2$ certainly all elements with simple form
\begin{align*}\phi=\left(\begin{array}{cccccccccc}
x_{1} & x_{2} & \cdots & x_{k-1} & x_{k} & x_{k+1} & \cdots & x_{k+m-1}&x_{k+m} \\
x_{1} & x_{2} & \cdots & x_{k-1} & x_{k+1} & x_{k+2} & \cdots & x_{k+m}&ax_k& \end{array}\right. & \\
\left.\begin{array}{ccc}
x_{k+m+1}&\cdots &x_{r} \\
x_{k+m+1}&\cdots &x_r\end{array}\right), &
\end{align*}
where $k\geq 1, m\geq 0, a\in G$, lie in $P$. In particular, we can choose
\begin{align*}\bold{r}_{l_0}=\left(\begin{array}{cccccccccccccccc}
x_{1} & x_{2} & \cdots  & \cdots & x_{k+m} & x_{k+m+1} & x_{k+m+2}& \cdots \\
x_{1} & x_{2} & \cdots  & \cdots & x_{k+m} & ax_{k} & x_{k+m+1}& \cdots & \end{array}\right. &\\
\left.\begin{array}{ccccc}
 \cdots & x_{r+1} & x_{r+2} &\cdots &x_n \\
 \cdots & x_{r} & x_1 &\cdots &x_1\end{array}\right), &
 \end{align*}
and $\mu_0=(1,\cdots,k-1, k+1, \cdots, r+1)$ to give $\bold{p}_{\mu_0 l_0}=\bold{q}_{\mu_0}\bold{r}_{l_0}=\phi$. We now proceed to show that if $\mathbf{p}_{\lambda i}=\phi\neq \varepsilon$, then $(i, \lambda)$
is connected to $(j, \mu_0)$ for some $j\in I$ and hence to $(l_0, \mu_0)$.

\begin{lem}\label{lem:simple}  Let $\varepsilon\neq\phi$ be as above and suppose that $\phi=\mathbf{p}_{\lambda i}\in H$ where  $\lambda=(u_1,\cdots,u_r)$ and $\mathbf{r}_i$ lies in district $(l_1,\cdots,l_r)$. Then $(i, \lambda)$ is connected to $(j, \mu_0)$
for some $j\in I$.
\end{lem}
\begin{proof}
Notice that as $\phi=\mathbf{p}_{\lambda i}$, we have $x_{u_k}r_i=x_k\phi,$ so that $x_{u_k}r_i=x_{k+1}$ if $m>0,$ and so $u_k\geq l_{k+1}>l_k$ by Lemma \ref{property of r_{i}}; or if $m=0$ and $a\neq 1_G$, $x_{u_k}r_i=ax_{k}$ so that $u_k>l_k$ by Lemma \ref{property of r_{i}} again. Further, from the constraints on $(l_1, \cdots, l_r)$ it follows that $$l_1<l_2<\cdots <l_{k-1}<l_k<u_k.$$

We first ensure that $(i, \lambda)$ is connected to some $(j, \kappa)$ where
$\kappa= (1,\hdots, k-1,u_k,\hdots,u_r)$,  by induction on $(l_1, \cdots, l_{k-1})\in [1, n]^r$ under the lexicographic order.

If $(l_1, \cdots, l_{k-1})=(1, \cdots, k-1),$ then clearly $(i, \lambda)=(i, \kappa)$. Suppose now that $(l_1, \cdots,
l_{k-1})>(1, \cdots, k-1)$ and the result is true for all $(l_1', \cdots, l_{k-1}')\in [1, n]^r$ where $(l_1',\cdots,
l_{k-1}')<(l_1, \cdots, l_{k-1}),$ namely, if $\mathbf{p}_{\eta l}=\phi$ with $\mathbf{r}_l$ in district $(l_1',
\cdots, l_r')$, then $(l, \eta)$ is connected to some $(j, \kappa)$.

By putting $\nu=(l_1,\cdots,l_{k-1},u_k,\cdots,u_r)$ we have $\mathbf{p}_{\nu i}=\mathbf{p}_{\lambda i}.$ Since we have
$(l_1, \cdots, l_{k-1})>(1, \cdots, k-1)$, there must be a $t\in (l_s,l_{s+1})\cap S(\nu,i)$ for some $s\in [0,k-2]$,
where $l_0=0.$ We can use Step D to move $l_{s+1}$ down to $t$, obtaining $\mathbf{r}_p$ in district $(l_1,\hdots,
l_s,t,l_{s+2},\hdots,l_r)$ such that $\mathbf{p}_{\nu i}=\mathbf{p}_{\nu p}.$ Clearly $$(l_1, \cdots, l_s, t, l_{s+2},
\cdots, l_{k-1})<(l_1, \cdots, l_{k-1}),$$ so that by induction $(p, \nu)$ (and hence $(i, \lambda)$) is connected to
some $(j, \kappa).$


We now proceed via induction on $(u_k,\hdots,u_r)\in [k+1,n]^r$ under the lexicographic order to show that $(j, \kappa)$ is connected to some $(l, \mu_0)$. Clearly, this is true for $(u_k,\hdots,u_r)=(k+1, \cdots, r+1).$

Suppose that  $(u_k,\hdots,
u_r)>(k+1,\hdots, r+1)$, and the result is true for all $(v_k, \cdots, v_r)\in [k+1, n]^r$ where $(v_k, \cdots, v_r)<(u_1, \cdots, u_r)$. Then we define $\mathbf{r}_w$ by:
\[x_l\mathbf{r}_w=x_l,l\in [1,k],\, x_{u_l}\mathbf{r}_w=x_{u_l}\mathbf{r}_j, l\in [k,r]\mbox{ and }x_v\mathbf{r}_w=x_1\mbox{ for all other }x_v.\]
It is easy to see that $\mathbf{r}_w\in\Theta$,   $\mathbf{r}_w$ lies in district $$(1,2,\cdots,k,u_k,\cdots,
u_{k+m-1},u_{k+m+1},\cdots,u_r)$$ and $\mathbf{p}_{\kappa j}=\mathbf{p}_{\kappa w}$. Note that there must exist a $t<u_h$ for some
$h\in [k,r]$ such that $[t,u_h)\subseteq S(\kappa,w)$. By Step U$'$, we have that $(w, \kappa)$ is connected to $(v, \rho)$
where $\rho=(1,\hdots,k-1,u_k,\hdots, u_{h-1},t,u_{h+1},\hdots, u_r)$. Clearly, $$(u_k,\hdots,
u_{h-1},t,u_{h+1},\hdots, u_r)<(u_k, \cdots, u_{h-1}, u_h, u_{h+1}, \cdots, u_r),$$ so that by induction  $(v, \rho)$
is connected to $(l, \mu_0).$ The proof is completed.
\end{proof}

The following corollary is immediate from Lemma 4.1, Corollary~\ref{coro:connected} and Lemma \ref{lem:simple}.

\begin{coro}\label{coro:simple} Let $\mathbf{p}_{\lambda i}=\mathbf{p}_{\nu k}$ have simple form. Then $f_{i,\lambda}=f_{k,\nu}$.
\end{coro}

\section{Non-identity generators with arbitrary form}\label{sec:arbform}

Our aim here is to prove that for any $\alpha\in H$, if $i, j\in I$ and $\lambda, \mu\in \Lambda$ with
$\mathbf{p}_{\lambda i}=\mathbf{p}_{\mu j}=\alpha\in H$, then $f_{i, \lambda}=f_{j, \mu}$.  This property of $\alpha$
is called {\em consistency}. Notice that Corollary~\ref{coro:simple} tells us that all elements with simple form are
consistent.

Before we explain the strategy in this section, we run the following example by the reader, which shows that if $|G|>1$, we cannot immediately separate  an element  $\alpha\in \overline{H}$ into a product, $\beta\gamma$ or $\gamma\beta$, where $\beta$ is  essentially an element of $\mathcal{S}_r$, and $\overline{\gamma}$ is the identity in $\mathcal{S}_r$.

Let $1_G\neq a$, $n=6$ and $r=4$, so that $\alpha=\left(\begin{array}{cccccc}
x_{1} & x_{2} & x_{3} & x_{4}\\
x_{3} & ax_{2} & x_{4} & x_{1} \end{array}\right)\in H$. By
putting $$\mathbf{r}_i=\left(\begin{array}{cccccccccccccccc}
x_{1} & x_{2} & x_{3} & x_{4} & x_{5} & x_{6} \\
x_{1} & x_{2} & x_{3} & ax_{2} & x_{4} & x_{1} \end{array}\right)$$ and $\lambda=(3,4,5,6),$ clearly we have
$\mathbf{p}_{\lambda i}=\alpha$.

Next we argue that $i\in I$ and $\lambda\in \Lambda$ are unique such that $\mathbf{p}_{\lambda
i}=\alpha.$ Let $\mu=(u_1, u_2, u_3, u_4)$ and $\mathbf{r}_j$ lie in district $(l_1, l_2,l_3,l_4)$ with
$\mathbf{p}_{\mu j}=\alpha$; we show that $\mathbf{r}_j=\mathbf{r}_i$ and $\mu=\lambda.$ Since
$x_{u_1}\mathbf{r}_j=x_1\alpha=x_3$ by assumption, we must have $l_1<l_2<l_3\leq u_1,$ so that $u_1\geq 3$. As $3\leq
u_1<u_2<u_3<u_4\leq n=6$, we have $\mu=(u_1, u_2,u_3,u_4)=(3,4,5,6)=\lambda$, and $(l_1,l_2)=(1,2)$. Clearly then $\mathbf{r}_j=\mathbf{r}_i.$


Certainly  $\alpha=\gamma\beta=\beta\gamma$ with $$\gamma=\left(\begin{array}{cccccc}
x_{1} & x_{2} & x_{3} & x_{4}\\
x_{3} & x_{2} & x_{4} & x_{1} \end{array}\right), \beta=\left(\begin{array}{cccccc}
x_{1} & x_{2} & x_{3} & x_{4}\\
x_{1} & ax_{2} & x_{3} & x_{4} \end{array}\right).$$

Our question is, can we find a sub-matrix of $P$ with one of the following forms: $$\left(\begin{array}{cc}
\gamma & \alpha \\
\varepsilon & \beta \end{array}\right)\mbox{~or~}\left(\begin{array}{cc}
\beta & \alpha \\
\varepsilon& \gamma \end{array}\right).$$ Clearly, here the answer is in the negative, as it is easy to see from the
definition $\mathbf{r}_i$ that there does not exist $\nu\in \Lambda$ with $\mathbf{p}_{\nu i}=\beta$ or
$\mathbf{p}_{\nu i}=\gamma$.

Now it is time for us to explain our trick of  how to split an arbitrary element $\alpha$ in $H$ into a product of
elements with simple form (defined in the previous section), and moreover, how this splitting  matches the products of
generators $f_{i, \lambda}$ in $\overline{H}.$

Our main strategy is as follows. We introduce a notion of `rising point' of $\alpha\in H$. Now, given
$\mathbf{p}_{\lambda i}=\alpha$, we
 decompose
$\alpha$ as a product $\alpha=\beta\gamma$ {\em depending only on $\alpha$} such that $\gamma$ is an element with
simple form, $\beta=\mathbf{p}_{\lambda j}$ has a lower rising point than $\alpha$, $\gamma=\mathbf{p}_{\mu i}$ for
some $j\in I,\mu\in\Lambda$ such that our presentation gives $f_{i,\lambda}=f_{i,\mu}f_{j,\lambda}$.

\begin{defn}\label{defn:risingpoint} Let $\alpha\in H$.  We say that $\alpha$ has {\it rising point} $r+1$ if
$x_m\alpha=ax_r$ for some $m\in [1,r]$ and $a\neq 1_G$; otherwise, the rising point is $k\leq r$ if there exists a sequence $$1\leq i<j_{1}<j_{2}<\cdots<j_{r-k}\leq r$$ with
$$x_{i}\alpha=x_{k}, x_{j_{1}}\alpha=x_{k+1}, x_{j_{2}}\alpha=x_{k+2}, \cdots, x_{j_{r-k}}\alpha=x_{r}$$ and
such that if $l\in [1,r]$ with  $x_{l}\alpha=ax_{k-1}$, then if    $l<i$ we must have $a\neq 1_G$.
\end{defn}

We briefly outline an algorithm for computing the value of the rising point of an element $\alpha\in G\wr \mathcal{S}_r$, which should convince our readers that the rising point is uniquely determined by $\alpha.$

(1) First look at the unique $ax_r$ in the image of $\alpha$. If $a\neq 1_G$ then set $k=r+1.$

(2) Otherwise, look to the left of $x_r$ and see if $ax_{r-1}$ appears to the left in the image in the standard two-row representation. If it does, check the value of $a$ in $ax_{r-1}$. If $a=1$ then repeat the process of looking left.

(3) Carrying out this process eventually one of two things must happen, either

(I) we stop because we reach some $ax_{k-1}$ with $a\neq 1_G$. Then we say the rising point value is $k.$ Or

(II) we reach $ax_k=x_k$ and do not see $ax_{k-1}$ to the left so the process stops and the rising point value is $k.$

We refer our readers to \cite{Yang:2014} for various examples of computing the rising point value.

It is easy to see that the only element with rising point 1 is the identity of $H$, and elements with rising point 2 have either of the following two forms:

(i) $\alpha=\left(\begin{array}{cccccccccccccccc}
x_{1} & x_{2} & \cdots & x_{r} \\
ax_{1} & x_{2} & \cdots & x_{r} \end{array}\right)$, where $a\neq 1_G$;

(ii) $\alpha=\left(\begin{array}{cccccccccccccccc}
x_{1} & x_{2} & \cdots & x_{k-1} & x_{k} & x_{k+1} & \cdots & x_{r} \\
x_{2} & x_{3} & \cdots & x_{k} & ax_{1} &x_{k+1} & \cdots & x_{r} \end{array}\right)$, where $k\geq 2.$

Note that both of the above two forms are the so called simple forms; however, elements with simple form can certainly have rising point greater than 2, indeed, it can be $r+1$. From Lemma~\ref{lem:idgens} and Corollary~\ref{coro:simple} we immediately deduce:

\begin{coro}\label{coro:1or2} Let $\alpha\in H$ have rising point 1 or 2. Then $\alpha$ is consistent.
\end{coro}

Next, we will see how to decompose an element with a  rising point at least 3 into a product of an element with a lower rising point and an element with simple form.

\begin{lem} \label{decomposition lemma}
Let $\alpha\in H$ have rising point $ k\geq3$. Then $\alpha$ can be expressed as a product of some $\beta \in H$ with  rising point no more than $k-1$ and some $\gamma \in H$ with  simple form.
\end{lem}

\begin{proof} \noindent{\em Case (0)}
By definition of rising point, if $k=r+1$, then we have $x_m\alpha=ax_r$ for some $a\neq 1_G$ and $m\in [1,r]$. We define $$\gamma=\left(\begin{array}{cccccccccccccccc}
x_{1} & x_{2} & \cdots & x_{r-1} & x_{r} \\
x_{1} & x_{2} & \cdots & x_{r-1} & ax_{r} \end{array}\right)$$ and $\beta$ by $x_{m}\beta=x_{r}$ and for other $j\in [1,r]$, $x_{j}\beta=x_{j}\alpha$. Clearly, $\alpha=\beta\gamma$, $\gamma$ is a simple form, and $\beta$ has rising point no greater than $r$.

\medskip

 On the other hand, if $k\leq r$ there exists a sequence $$1\leq i<j_{1}<j_{2}\cdots <j_{r-k}\leq r$$ with $$x_{i}\alpha=x_{k}, x_{j_{1}}\alpha=x_{k+1}, x_{j_{2}}\alpha=x_{k+2}, \cdots, x_{j_{r-k}}\alpha=x_{r}$$
 such that if $l\in [1,r]$ with  $x_{l}\alpha=ax_{k-1}$, then if    $l<i$ we must have $a\neq 1_G$.
 We proceed by considering the following cases:

\noindent{\em Case (i)}
 If $l<i$, so that  $a\neq 1_G$, then define  $$\gamma=\left(\begin{array}{cccccccccccccccc}
x_{1} & x_{2} & \cdots & x_{k-2} & x_{k-1} & x_{k} & \cdots & x_{r} \\
x_{1} & x_{2} & \cdots & x_{k-2} & ax_{k-1} &x_{k} & \cdots & x_{r} \end{array}\right)$$
and put $\beta=\alpha\gamma^{-1}$. It is easy to check that $x_{l}\beta=x_l\alpha\gamma^{-1}=x_{k-1}$ and $x_{p}\beta =x_{p}\alpha$, for other $p\in [1,r]$.

\noindent{\em Case (ii)} If $i<l<j_1$, then define
$$\gamma=\left(\begin{array}{cccccccccccccccc}x_{1} & x_{2} & \cdots & x_{k-2} & x_{k-1} & x_{k} & x_{k+1}& \cdots & x_{r} \\
x_{1} & x_{2} & \cdots & x_{k-2} & x_{k} & ax_{k-1} & x_{k+1} & \cdots & x_{r} \end{array}\right)$$
and again, we put $\beta=\alpha\gamma^{-1}$. By easy calculation we have $$x_{i}\beta=x_{k-1}, x_{l}\beta=x_{k}, x_{j_{1}}\beta=x_{k+1},\cdots, x_{j_{r-k}}\beta=x_{r}$$ and for other $p\in [1,r]$, $x_{p}\beta=x_{p}\alpha$.

\noindent{\em Case (iii)} If  $j_{r-k}<l$, then define $$\gamma=\left(\begin{array}{cccccccccccccccc}
x_{1} & x_{2} & \cdots & x_{k-2} & x_{k-1} & x_{k} & x_{k+1}& \cdots & x_{r-1} & x_{r} \\
x_{1} & x_{2} & \cdots & x_{k-2} & x_{k} & x_{k+1} & x_{k+2} & \cdots & x_{r} & ax_{k-1}\end{array}\right)$$
and again, we define $\beta=\alpha\gamma^{-1}$.  It is easy to see that $$x_{i}\beta=x_{k-1}, x_{j_{1}}\beta=x_{k}, x_{j_{2}}\beta=x_{k+1},\cdots, x_{j_{r-k}}\beta=x_{r-1}, x_{l}\beta=x_{r}$$ and for other $p\in [1,r]$, $x_{p}\beta=x_{p}\alpha$.

\noindent{\em Case (iv)} If $j_{u}<l<j_{u+1}$ for some $u\in[1,r-k-1]$, then define
\begin{align*}
\gamma=\left(\begin{array}{ccccccccccccccc}
x_{1} & x_{2} & \cdots & x_{k-2} & x_{k-1} & x_{k} &\cdots&x_{k+u-1}&x_{k+u}&   \\
x_{1} & x_{2} & \cdots & x_{k-2} & x_{k} & x_{k+1} & \cdots& x_{k+u}&ax_{k-1} &
\end{array}\right. & \\
\left.\begin{array}{ccc}
x_{k+u+1}& \cdots & x_{r} \\
x_{k+u+1}&    \cdots   & x_{r}
\end{array}\right) &
\end{align*}
and again, we put $\beta=\alpha\gamma^{-1}$.  Then we have $$x_{i}\beta=x_{k-1},\ x_{j_1}\beta=x_k,\ \hdots\ ,
x_{j_u}\beta=x_{k+u-1},\ x_l\beta= x_{k+u},$$ $$ x_{j_{u+1}}\beta=x_{k+u+1},\ \hdots\ , x_{j_{r-k}}\beta=x_r$$ and for
other $p\in [1,r]$, $x_{p}\beta=x_{p}\alpha$.

In each of Cases $(i)-(iv)$ it is clear that  $\gamma$ has  simple form,  $\alpha=\beta\gamma$ and $\beta$ has a rising point no more than $k-1$.  The proof is completed.
\end{proof}

Note that in each of Cases $(ii)-(iv)$ of Lemma~\ref{decomposition lemma}, that is, where $i<l$, we have   $x_{p}\beta=x_{p}\alpha$ for all $p<i$.

\begin{lem}\label{lem:consind} Let $\alpha,\beta,\gamma\in H$ with $\alpha=\beta\gamma$ and $\beta,\gamma$ consistent.
Suppose that whenever $\alpha=\mathbf{p}_{\lambda j}$, we can find $(t,\lambda), (j,\mu)\in K$ with
$\beta=\mathbf{p}_{\lambda t}, \gamma=\mathbf{p}_{\mu j}$ and $f_{j,\lambda}=f_{j,\mu}f_{t,\lambda}$. Then $\alpha$ is
consistent.
\end{lem}

\begin{proof} Let $\alpha,\beta,\gamma$ satisfy the hypotheses of the lemma. If $\alpha=\mathbf{p}_{\lambda j}=\mathbf{p}_{\lambda' j'}$, then
by assumption we can find  $(t,\lambda), (j,\mu), (t',\lambda'), (j',\mu')\in K$ with $\beta=\mathbf{p}_{\lambda
t}=\mathbf{p}_{\lambda' t'}, \gamma=\mathbf{p}_{\mu j}=\mathbf{p}_{\mu' j'}$,  $f_{j,\lambda}=f_{j,\mu}f_{t,\lambda}$
and $f_{j',\lambda'}=f_{j',\mu'}f_{t',\lambda'}$. The result now follows from the consistency of $\beta$ and $\gamma$.
\end{proof}

\begin{prop}\label{prop:eqofgens} Every $\alpha\in P$ is consistent. Further, if $\alpha=\mathbf{p}_{\lambda j}$ then $f_{j, \lambda}$ is equal in $\overline{H}$ to a product
$f_{i_1,\lambda_1}\cdots f_{i_k,\lambda_k},$ where $\mathbf{p}_{\lambda_t,i_t}$ is an element with simple form, $t\in
[1,k]$.
\end{prop}

\begin{proof}
We proceed by induction on the rising point of $\alpha$. If $\alpha$ has rising point $1$ or $2$, and
$\mathbf{p}_{\lambda i}=\alpha$, then the result is true by Corollary~\ref{coro:1or2} and the comments preceding it.
Suppose for induction that the rising point of $\alpha$ is $k\geq 3$, and the result is true for all $\beta\in H$ with
rising point strictly less than $k$ and all $f_{i,\mu}\in F$ where $\mathbf{p}_{\mu i}=\beta$.

 We proceed on a case by case basis, using  $\gamma$ and $\beta$ as defined in Lemma~\ref{decomposition lemma}. Since $\gamma$ has simple form, it is consistent by Corollary~\ref{coro:simple} and as $\beta$ has rising point strictly less
than $k$, $\beta$ is consistent by our inductive hypothesis.

Suppose that $\alpha=\mathbf{p}_{\lambda j}$ where $\lambda=(u_1,\hdots, u_r)$ and $\mathbf{r}_j$ lies in district
$(l_1,\hdots,l_r)$.

\medskip

\noindent{\em Case (0)} If $k=r+1$, then we have $x_m\alpha=ax_r$ for some $a\neq 1_G$.
 We now define $\mathbf{r}_{t}$ by $x_{u_{m}}\mathbf{r}_{t}=x_{r}$ and $x_{s}\mathbf{r}_{t}=x_{s}\mathbf{r}_{j}$, for other $s\in [1,n]$. As $x_{u_{m}}\mathbf{r}_{j}=ax_{r}$, it is easy to see that $\mathbf{r}_{t}\in\Theta$.  Notice that $l_{r-1}<l_r=l_{m\overline{\alpha}}<u_m$. Then by setting $\mu=(1,l_{2},\cdots,l_{r-1},u_{m})$ we have $$\left(\begin{array}{cc}
\mathbf{p}_{\lambda t} & \mathbf{p}_{\lambda j}\\
\mathbf{p}_{\mu t} & \mathbf{p}_{\mu j} \end{array}\right)=\left(\begin{array}{cc}
\beta & \alpha\\
\epsilon & \gamma \end{array}\right)$$
and our presentation gives $f_{j,\lambda}=f_{j,\mu}f_{t,\lambda}$.

\medskip

We  now suppose that $k\leq r$. By definition of
 rising point there exists a sequence $$1\leq i<j_{1}<j_{2}\cdots <j_{r-k}\leq r$$ such that $$x_{i}\alpha=x_{k}, x_{j_{1}}\alpha=x_{k+1}, x_{j_{2}}\alpha=x_{k+2}, \cdots, x_{j_{r-k}}\alpha=x_{r}$$
 such that if $l\in [1,r]$ with  $x_{l}\alpha=ax_{k-1}$, then if    $l<i$ we must have $a\neq 1_G$.

We  consider the following cases:

\medskip

\noindent{\em Case (i)} If $l<i$ we define $\mathbf{r}_{t}$ by $x_{u_{l}}\mathbf{r}_{t}=x_{k-1}$ and for other $p\in
[1,n]$, $x_{p}\mathbf{r}_{t}=x_{p}\mathbf{r}_{j}$. As by assumption $x_{u_{l}}\mathbf{r}_{j}=x_{l}\alpha=ax_{k-1}$,
clearly $\mathbf{r}_{t}\in\Theta $.  Then by putting
$$\mu=(1,l_{2},\cdots,l_{k-2},u_{l},u_{i},u_{j_{1}},\cdots,u_{j_{r-k}})$$ we have $$\left(\begin{array}{cc}
\mathbf{p}_{\lambda t} & \mathbf{p}_{\lambda j}\\
\mathbf{p}_{\mu t} & \mathbf{p}_{\mu j} \end{array}\right)=\left(\begin{array}{cc}
\beta & \alpha\\
\epsilon & \gamma \end{array}\right)$$
which implies $f_{j,\lambda}=f_{j,\mu}f_{t,\lambda}.$

\medskip

\noindent{\em Case (ii)}  If $i<l<j_{1}$ we define $\mathbf{r}_s$ by
\[x_p\mathbf{r}_s=x_p\mathbf{r}_j\mbox{ for }p<u_i, \, x_{u_w}\mathbf{r}_s=x_w\beta\mbox{ for }i\leq w\leq r\]
\[\mbox{ and }x_{v}\mathbf{r}_{s}=x_{1}\mbox{ for all other }v\in [1,n].\] We must argue that $\mathbf{r}_s\in\Theta$.
Note that from the comment following Lemma~\ref{decomposition lemma}, for any $v<i$ we have that
\[x_{u_v}\mathbf{r}_s=x_{u_v}\mathbf{r}_j=x_v\alpha=x_v\beta,\]
so that in particular, rank $\mathbf{r}_s=r$. Further,
$$x_{u_i}\mathbf{r}_s=x_i\beta=x_{k-1},\ x_{u_l}\mathbf{r}_s=x_l\beta=x_k,$$
$$x_{u_{j_1}}\mathbf{r}_s=x_{j_1}\beta=x_{k+1},\ \hdots\ ,  x_{u_{j_{r-k}}}\mathbf{r}_s=x_{j_{r-k}}\beta=x_{r}$$ so that $$\langle x_{u_i},x_{ u_l},x_{u_{j_1}},
\cdots,x_{u_{j_{r-k}}}\rangle\mathbf{r}_s=\langle x_{k-1},x_{k},\cdots,x_r\rangle.$$ Thus for any $v\neq \{i, l, j_1,
\cdots,j_{r-k}\},$ $x_{u_v}\mathbf{r}_s=x_v\beta\in \langle x_1, \cdots,  x_{k-2}\rangle.$

As $x_{u_i}\mathbf{r}_j=x_k$, we have $1=l_1<l_2<\cdots<l_{k-1}<l_k\leq u_i.$ Let $h$ be the largest number with
$$1=l_1<l_2<\cdots<l_{k-1}<l_k<l_{k+1}<\cdots<l_{(k-1)+h}<  u_i.$$  Clearly here we have $h\in [0,r-k+1]$.  Now we
claim that $\mathbf{r}_s\in \Theta$ and lies in district $$(l_1, l_2, \cdots, l_{(k-1)+h}, u_{j_{h}}, u_{j_{h+1}},
\cdots, u_{j_{r-k}}).$$
 To simplify our notation we put $$(l_1, l_2, \cdots, l_{(k-1)+h}, u_{j_{h}}, u_{j_{h+1}}, \cdots, u_{j_{r-k}})=(z_1, z_2, \cdots, z_{(k-1)+h}, z_{k+h}, \cdots, z_r),$$
 where $j_0=l$. Clearly, by the definition of $\mathbf{r}_s$, we have $x_{z_v}\mathbf{r}_s=x_v$ for all $v\in [1,r].$ Hence, to show $\mathbf{r}_s\in \Theta$, by the definition we only need to argue that for any $m\in [1,n]$ and $b\in G$, $x_m\mathbf{r}_s=bx_t$ implies $m\geq z_t.$

  Suppose that $t\in [1,(k-1)+h]$, so that  $z_t=l_t<u_i$. If $m<z_t$, then  from the definition of $\mathbf{r}_s$
  we have  $x_m\mathbf{r}_s=x_m\mathbf{r}_j,$ so that $x_m\mathbf{r}_j=bx_t.$ As $\mathbf{r}_j\in \Theta$ and $x_{l_t}\mathbf{r}_j=x_t,$ we have $z_t=l_t\leq m,$ a contradiction, and we deduce that  $m\geq z_t.$

  Suppose now that $t\in [k+h, r].$ Note that $m\geq u_i;$ because, if $m<u_i,$ then $x_m\mathbf{r}_j=x_m\mathbf{r}_s=bx_t.$ As $\mathbf{r}_j\in \Theta$, $l_t\leq m<u_i$ and so $t\leq (k-1)+h$, a contradiction.  Thus $m\geq u_i$. Now,  by the definition of $\mathbf{r}_s$, we know there is exactly one possibility that $x_m\mathbf{r}_s=bx_t$ with $t\in [k+h, r],$ that is, $x_{z_t}\mathbf{r}_s=x_t,$ so that $m=z_t$ and $b=1.$ Thus $\mathbf{r}_s\in \Theta$.

 Now set $$\eta=(1,l_{2},\cdots,l_{k-2},u_{i},u_{l},u_{j_{1}},\cdots,u_{j_{r-k}})$$ then we have
$$\left(\begin{array}{cc}
\mathbf{p}_{\lambda s} & \mathbf{p}_{\lambda j}\\
\mathbf{p}_{\eta s} & \mathbf{p}_{\eta j} \end{array}\right)=\left(\begin{array}{cc}
\beta & \alpha\\
\epsilon & \gamma \end{array}\right),$$ which
implies $f_{j,\lambda}=f_{j,\eta}f_{s,\lambda}.$

\medskip

\noindent{\em Case (iii)} If $j_{r-k}<l$, then, defining $\mathbf{r}_s$ as in {\em Case (ii)}, a similar argument gives that
$\mathbf{r}_s\in\Theta$ and $x_{u_v}\mathbf{r}_s=x_v\beta$ for all $v\in [1,r]$ (of course here $\beta$ is defined
differently to that given in {\em Case (ii)} and the district of $\mathbf{r}_s$ will have a different appearance.). Moreover, by setting
$$\delta=(1,l_{2},\cdots,l_{k-2},u_{i},u_{j_{1}},\cdots,u_{j_{r-k}},u_{l})$$ we have
$$\left(\begin{array}{cc}
\mathbf{p}_{\lambda s} & \mathbf{p}_{\lambda j}\\
\mathbf{p}_{\delta s} & \mathbf{p}_{\delta j} \end{array}\right)=\left(\begin{array}{cc}
\beta & \alpha\\
\epsilon & \gamma \end{array}\right).$$
implying  $f_{j,\lambda}=f_{j,\delta}f_{s,\lambda}$.

\medskip

\noindent{\em Case (iv)} If $j_{u}<l<j_{u+1}$ for some $u\in[1,r-k-1]$, then again by defining $\mathbf{r}_s$ as in
{\em Case (ii)}, we have $\mathbf{r}_s\in\Theta$ and $x_{u_v}\mathbf{r}_s=x_v\beta$ for all $v\in [1,r]$. Take
 $$\sigma=(1,l_2,\cdots,l_{k-2},u_i,u_{j_1}, u_{j_u}, u_l, u_{j_{u+1}}, \cdots, u_{j_{r-k}}).$$ Then we have
$$\left(\begin{array}{cc}
\mathbf{p}_{\lambda s} & \mathbf{p}_{\lambda j}\\
\mathbf{p}_{\sigma s} & \mathbf{p}_{\sigma j} \end{array}\right)=\left(\begin{array}{cc}
\beta & \alpha\\
\epsilon & \gamma \end{array}\right)$$
so that $f_{j,\lambda}=f_{j,\sigma}f_{s,\lambda}$.

In each of the cases above, the consistency of $\alpha$ follows from Lemma~\ref{lem:consind}. The result now follows by induction.
\end{proof}

In view of Lemma \ref{lem:consind}, we can now denote all generators $f_{i,\lambda}$ with $\mathbf{p}_{\lambda
i}=\alpha$ by $f_{\alpha}$, where $(i,\lambda)\in K$.

\section{The main theorem}\label{sec:main}

Our eventual aim is to show that $\overline{H}$ is isomorphic to $H$ and hence to the wreath product $G\wr \mathcal{S}_r$. With this  in mind,
given the knowledge we have gathered concerning the generators $f_{i,\lambda}$, we first specialise the general presentation given in Theorem~\ref{them:pres} to our specific situation.

We will say that for $\phi,\varphi,\psi,\sigma\in P$ the quadruple $(\phi,\varphi,\psi,\sigma)$ is {\em singular} if
$\phi^{-1}\psi=\varphi^{-1}\sigma$ and we can find $i,j\in I, \lambda,\mu\in \Lambda$ with $\phi=\mathbf{p}_{\lambda
i}, \varphi =\mathbf{p}_{\mu i}, \psi=\mathbf{p}_{\lambda j}$ and $\sigma=\mathbf{p}_{\mu j}$.


In the sequel, we denote the free group on a set $X$ by $\widetilde{X}$. For convenience, we use, for example, the same symbol
$f_{i,\lambda}$ for an element of $\widetilde{F}$ and $\overline{H}$. We hope that the context will prevent ambiguities from arising.

\begin{lem} \label{presentation}
Let $\overline{\overline{H}}$ be the group given by the presentation $\mathcal{Q}=\langle S: \Gamma\rangle$ with generators: $$S=\{f_{\phi}:~ \phi \in P\}$$ and with the defining relations $\Gamma:$

$(P1)$ $f_\phi^{-1}f_\varphi=f_\psi^{-1}f_\sigma $ where $(\phi,\varphi,\psi,\sigma)$ is singular;

$(P2)$ $f_{\epsilon}=1.$\\
Then $\overline{\overline{H}}$ is isomorphic to $\overline{H}$.
\end{lem}

\begin{proof}
From Theorem \ref{them:pres}, we  know that $\overline{H}$ is given by the presentation  $\mathcal{P}=\langle F:
\Sigma\rangle$,  where $F=\{f_{i,\lambda}:~ (i,\lambda)\in K\}$ and $\Sigma$ is the set of relations as defined in (R1)
-- (R3), and  where the function $\omega$ and the Schreier system $\{ \mathbf{h}_{\lambda}:\lambda\in \Lambda\}$ are
fixed as in Section~\ref{sec:pres}. Note that (R3) is reformulated in Corollary~\ref{lem:ssindr}.

By freeness of the generators we may define a morphism $\boldsymbol{\theta}:\widetilde{F}\rightarrow
\overline{\overline{H}}$ by $f_{i,\lambda}\boldsymbol{\theta}= f_\phi$, where $\phi=\mathbf{p}_{\lambda i}$. We show
that $\Sigma\subseteq \ker\boldsymbol{\theta}$. It is clear from (P1) that relations of the form (R3) lie in
$\ker{\boldsymbol\theta}$.

Suppose first that $\mathbf{h}_{\lambda} \varepsilon_{i\mu}=\mathbf{h}_{\mu}$ in $\overline{E}^*$. Then $\epsilon \mathbf{h}_{\lambda}
\varepsilon_{i\mu}=\epsilon \mathbf{h}_{\mu}$ in $\en F_n(G)$, so that from Lemma~\ref{lem:righttrans},
$\mathbf{q}_{\lambda}\varepsilon_{i\mu}=\mathbf{q}_{\mu}$. Hence
$\mathbf{q}_{\mu}\mathbf{r}_i=\mathbf{q}_{\lambda}\varepsilon_{i\mu}\mathbf{r}_i=\mathbf{q}_{\lambda}\mathbf{r}_i$, so that
$\mathbf{p}_{\mu i}=\mathbf{p}_{\lambda i}$ and $f_{i,\lambda}{\boldsymbol \theta}=f_{i,\mu}{\boldsymbol \theta}$. Now
suppose that $i\in I$; we have remarked that $\mathbf{p}_{\omega(i)i}=\epsilon$, so that $f_{i,\omega(i)}{\boldsymbol
\theta}=f_\epsilon=1{\boldsymbol\theta}$.

We have shown that $\Sigma\subseteq \ker{\boldsymbol \theta}$ and so  there exists a morphism $\overline{{\boldsymbol
\theta}}: {\overline{H}}\rightarrow {\overline{\overline{H}}}$ such that $f_{i, \lambda}\overline{{\boldsymbol
\theta}}=f_{\phi}$ where $\phi=\mathbf{p}_{\lambda i}$.

Conversely, we define a map ${\boldsymbol \psi}: \widetilde{S}\rightarrow {\overline{H}}$ by $f_\phi{\boldsymbol
\psi}=f_{i,\lambda}$, where $\phi=\mathbf{p}_{\lambda i}$. By Lemma~\ref{lem:consind}, ${\boldsymbol \psi}$ is well
defined. Since $f_\epsilon{\boldsymbol \psi}=f_{i,\lambda}$ where $\mathbf{p}_{\lambda i}=\epsilon$, we have
$f_\epsilon{\boldsymbol \psi}=1_{\overline{H}}$ by Lemma~\ref{lem:idgens}. Clearly relations (P1) lie in
$\ker{\boldsymbol \psi}$, so that $\Gamma\subseteq \ker{\boldsymbol \psi}$. Consequently, there is a morphism
$\overline{{\boldsymbol \psi}}: \overline{\overline{H}}\rightarrow \overline{H}$ such that
$f_\phi\overline{{\boldsymbol \psi}}=f_{i,\lambda}$, where $\phi=\mathbf{p}_{\lambda i}$.

It is clear that $\overline{{\boldsymbol \theta}}\,\overline{{\boldsymbol \psi}}$ and
$\overline{{\boldsymbol\psi}}\,\overline{{\boldsymbol\theta}}$ are, respectively, the identity maps on the generators of $\overline{H}$ and $ \overline{\overline{H}}$, respectively. It follows immediately that they are mutually inverse isomorphisms.
\end{proof}

We now  recall  the presentation of $G\wr \mathcal{S}_r$ obtained by Lavers \cite{Lavers:1998}. In fact, we translate his presentation to one for
our group $H$.

We begin by defining  the following elements of $H$:
for $a\in G$ and for $ 1\leq i\leq r$ we put $$\iota_{a,i}=\left(\begin{array}{cccccccccccccccc}
x_{1} & \cdots & x_{i-1} & x_{i}  & x_{i+1} & \cdots & x_{r} \\
x_{1} & \cdots & x_{i-1} & ax_{i} & x_{i+1} & \cdots & x_{r}\end{array}\right);$$  for $1\leq k\leq r-1$ we put

\begin{align*}
 (k\ k+1\cdots k+m)=\left(\begin{array}{ccccccccccccc}
x_{1} & \cdots & x_{k-1} & x_{k} & \cdots & x_{k+m-1} & x_{k+m} \\
x_{1} & \cdots & x_{k-1} & x_{k+1} & \cdots & x_{k+m} & x_{k} \end{array}\right. & \\
\left.\begin{array}{ccc}
x_{k+m+1} &\cdots & x_{r} \\
x_{k+m+1} &\cdots & x_{r}\end{array}\right)
\end{align*}
and we denote $(k\ k+1)$ by $\tau_k$.


It is clear that $G^r$ has presentation
$\mathcal{V}=\langle Z:~\Pi \rangle$, with generators
\[Z=\{ \iota_{a,i}:i\in [1,r],a\in G\}\]
and defining relations $\Pi$ consisting of (W4) and (W5) below. Using a standard presentation for $\mathcal{S}_r$  with generators the transpositions $\tau$ and relations $(W1)$, $(W2)$ and $(W3)$, we employ the recipe of \cite{Lavers:1998} to obtain:

\begin{lem}\label{presentation of wreath product}
The group $H$
 has a presentation $\mathcal{U}=\langle Y:~\Upsilon \rangle$, with generators
$$Y=\{\tau_i, \iota_{a,j}: ~1\leq i\leq r-1,\ 1\leq j\leq r, a\in G\}$$
and  defining relations $\Upsilon$:

$(W1)$ $\tau_i\tau_i=1$, $1\le i\leq r-1;$

$(W2)$ $\tau_i\tau_j=\tau_j\tau_i$, $j\pm 1\neq i\neq j$;

$(W3)$ $\tau_i\tau_{i+1}\tau_i=\tau_{i+1}\tau_i\tau_{i+1}$, $1\leq i\leq r-2$;

$(W4)$ $\iota_{a,i}\iota_{b,j}=\iota_{b,j}\iota_{a,i}$, $a,b\in G$ and $1\leq i\neq j\leq r$;

$(W5)$ $\iota_{a,i}\iota_{b,i}=\iota_{ab,i}$, $1\leq i\leq r$ and $a,b\in G$;

$(W6)$ $\iota_{a,i}\tau_j=\tau_j\iota_{a,i},$ $1\leq i\neq j, j+1\leq r$;

$(W7)$ $\iota_{a,i}\tau_i=\tau_i\iota_{a, i+1},$ $1\leq i\leq r-1$ and $a\in G.$
\end{lem}

Now we turn to our maximal subgroup $\overline{H}.$ From Lemma~\ref{presentation}, we know that
$\overline{H}$ is isomorphic to $ \overline{\overline{H}}$, and it follows from the definition of the isomorphism and  Proposition~\ref{prop:eqofgens} that
$$ \overline{\overline{H}}=\langle f_{\alpha}:\  \alpha \mbox{\ has simple\ form}\rangle.$$
We now simplify our generators further. For ease in the remainder of the paper, it is convenient to use the following convention:
for $u,v\in [1,r+2]$ with $u<v$, we denote by $\neg (u,v)$ the $r$-tuple
\[(1,\cdots, u-1,u+1,\hdots, v-1,v+1,\cdots, r+2).\]

\begin{lem}\label{further decomposition}
Consider the element $$\alpha=\left(\begin{array}{cccccccccccccccc}
x_{1} & \cdots & x_{k-1} & x_{k} & \cdots & x_{k+m-1} & x_{k+m}& x_{k+m+1} &\cdots & x_{r} \\
x_{1} & \cdots & x_{k-1} & x_{k+1} & \cdots & x_{k+m} & ax_{k} & x_{k+m+1} &\cdots & x_{r}\end{array}\right)$$
in simple form, where $m\geq 1.$
Then $f_{\alpha}=f_{\gamma}f_{\beta}$ in $ \overline{\overline{H}}$, where $\beta=\iota_{a, k+m}$ and $\gamma=(k\ k+1\ \cdots\ k+m).$
\end{lem}
\begin{proof}
Define $\mathbf{r}_t$ by
\begin{align*}
\left(\begin{array}{cccccccccccc}
x_{1} & \cdots & x_{k-1} & x_{k} & x_{k+1} &\cdots & x_{k+m} & x_{k+m+1}& x_{k+m+2} & x_{k+m+3} &\cdots \\
x_{1} & \cdots & x_{k-1} & x_{k} & x_{k+1} &\cdots & x_{k+m} & x_{k} & ax_{k} & x_{k+m+1} &\cdots \end{array}\right. &
\\
\left.\begin{array}{ccccc}
\cdots & x_{r+2} & x_{r+3} & \cdots & x_n\\
\cdots & x_{r} & x_1 &\cdots & x_1\end{array}\right). &
\end{align*}
Let $\lambda=\neg(k,k+m+1)$ and $\mu=\neg(k,k+m+2)$. Then $\mathbf{p}_{\lambda t}=\alpha$ and $\mathbf{p}_{\mu
t}=\gamma.$

Next we define $\mathbf{r}_s$ by
\begin{align*}
\left(\begin{array}{cccccccccc}
x_{1} & \cdots & x_{k-1} & x_{k} & x_{k+1} &\cdots & x_{k+m} & x_{k+m+1}& x_{k+m+2}\\
x_{1} & \cdots & x_{k-1} & x_{k-1} & x_{k} &\cdots & x_{k+m-1} & x_{k+m} & ax_{k+m} \end{array}\right. & \\
\left.\begin{array}{cccccc}
x_{k+m+3} &\cdots & x_{r+2} & x_{r+3} & \cdots & x_n\\
x_{k+m+1} &\cdots & x_{r} & x_1 &\cdots & x_1\end{array}\right).
\end{align*}
Then $\mathbf{p}_{\lambda s}=\beta$ and $\mathbf{p}_{\mu s}=\epsilon$. Notice that $\alpha=\beta\gamma$ and
$$\left(\begin{array}{cc}
\mathbf{p}_{\lambda s} & \mathbf{p}_{\lambda t}\\
\mathbf{p}_{\mu s} & \mathbf{p}_{\mu t}\end{array}\right)=\left(\begin{array}{cc}
\beta & \alpha\\
\epsilon & \gamma \end{array}\right)$$ which implies $f_{\alpha}=f_{\gamma}f_{\beta}.$
\end{proof}


\begin{lem}\label{transposition}
Let $\alpha=(k \ k+1 \ \cdots  \ k+m)$, where $m\geq 1.$ Then $$f_\alpha=f_{\tau_k}f_{\tau_{k+1}}\cdots
f_{\tau_{k+m-1}}$$ in $ \overline{\overline{H}}$.
\end{lem}
\begin{proof}
We proceed by induction on $m$: clearly the result is true for $m=1$. Assume now that $m\geq 2$,
$\alpha=(k \ k+1 \ \cdots  \ k+m)$ and that $f_{(k \ k+1 \ \cdots  \ k+s)}=f_{\tau_k}f_{\tau_{k+1}}\cdots f_{\tau_{k+s-1}}$  for any $s<m$. It is easy to check that
$\alpha=\tau_{k+m-1} \gamma$, where
\[ \gamma=(k\ k+1\ \cdots \ k+m-1).\]

Now we define $\mathbf{r}_j$ by
\begin{align*}
\left(\begin{array}{ccccccccccccc}
x_{1} & \cdots & x_{k+m-1} & x_{k+m}& x_{k+m+1} & x_{k+m+2} & x_{k+m+3} &\cdots & x_{r+2} \\
x_{1} & \cdots & x_{k+m-1} & x_{k} & x_{k+m} & x_{k} & x_{k+m+1} & \cdots &  x_{r}
\end{array}\right. & \\
\left.\begin{array}{ccc}
x_{r+3} & \cdots & x_n\\
x_{1} & \cdots & x_1\end{array}\right). &
\end{align*}
Let $\lambda=\neg(k,k+m)$ and $\mu=\neg(k,k+m+2)$. Then $\mathbf{p}_{\lambda
j}=\alpha$ and $\mathbf{p}_{\mu j}=\gamma$.

Next we define $\mathbf{r}_l$ by
\begin{align*}
\left(\begin{array}{cccccccccccc}
x_{1} & \cdots & x_{k-1} & x_{k} & x_{k+1} &\cdots & x_{k+m}& x_{k+m+1} & x_{k+m+2} & x_{k+m+3} & \cdots  \\
x_{1} & \cdots & x_{k-1} & x_{k} & x_{k} & \cdots & x_{k+m-1} & x_{k+m} & x_{k+m-1} & x_{k+m+1} & \cdots
\end{array}\right. & \\
\left.\begin{array}{ccccc}
\cdots & x_{r+2} & x_{r+3} & \cdots & x_n\\
\cdots & x_{r} & x_{1} & \cdots &  x_{1}\end{array}\right). &
\end{align*}
 Then $\mathbf{p}_{\lambda l}=\tau_{k+m-1}$ and $\mathbf{p}_{\mu
l}=\epsilon$. Thus we have $$\left(\begin{array}{cc}
\mathbf{p}_{\lambda l} & \mathbf{p}_{\lambda j}\\
\mathbf{p}_{\mu l} & \mathbf{p}_{\mu j}\end{array}\right)=\left(\begin{array}{cc}
\tau_{k+m-1} & \alpha\\
\epsilon & \gamma \end{array}\right)$$ implying $f_{\alpha}=f_{\gamma}f_{\tau_{k+m-1}}$ and so  $f_\alpha=f_{\tau_k}\cdots f_{\tau_{k+m-1}}$, using our inductive hypothesis applied to $\gamma$.
\end{proof}

It follows from Lemmas~ \ref{further decomposition} and \ref{transposition} that $$ \overline{\overline{H}}=\langle f_{\tau_i}, f_{\iota_{a,j}}:~ 1\leq i\leq r-1, 1\leq j\leq r, a\in G\rangle.$$ Now it is time for us to find a series of relations satisfied by these generators. These correspond to those in Lemma \ref {presentation of wreath product}, with the exception of a twist in (W5).

\begin{lem}\label{(i i+1)}
For all $i\in [1,r-1],$ $f_{\tau_i}f_{\tau_i}=1,$ and so $f_{\tau_i}^{-1}=f_{\tau_i}.$
\end{lem}
\begin{proof}
Notice that $\tau_i\tau_i=\epsilon$. First we define $\mathbf{r}_s$ by $$\left(\begin{array}{cccccccccccccccc}
x_{1} & \cdots & x_{i-1} & x_i & x_{i+1} & x_{i+2}& x_{i+3} & x_{i+4} & \cdots & x_{r+2} & x_{r+3} & \cdots & x_n\\
x_{1} & \cdots & x_{i-1} & x_i &  x_{i} & x_{i+1} & x_{i} & x_{i+2} & \cdots &  x_{r} & x_{1} & \cdots &
x_1\end{array}\right).$$ Let $\lambda=\neg (i,i+1)$ and $\mu=\neg (i,i+3)$. Then $\mathbf{p}_{\lambda s}=\tau_i$ and
$\mathbf{p}_{\mu s}=\epsilon$.

Next, we define $\mathbf{r}_t$ by $$\left(\begin{array}{cccccccccccccccc}
x_{1} & \cdots & x_{i-1} & x_i & x_{i+1} & x_{i+2}& x_{i+3} & \cdots & x_{r+2} & x_{r+3} & \cdots & x_n\\
x_{1} & \cdots & x_{i-1} & x_i &  x_{i+1} & x_{i} & x_{i+1} & \cdots &  x_{r} & x_{1} & \cdots &
x_1\end{array}\right).$$ Then $\mathbf{p}_{\lambda t}=\epsilon$ and $\mathbf{p}_{\mu t}=\tau_i$, so
$$\left(\begin{array}{cc}
\mathbf{p}_{\lambda s} & \mathbf{p}_{\lambda t}\\
\mathbf{p}_{\mu s} & \mathbf{p}_{\mu t}\end{array}\right)=\left(\begin{array}{cc}
\tau_i & \epsilon\\
\epsilon & \tau_i \end{array}\right)$$ which implies $f_{\tau_i}f_{\tau_i}=1$.
\end{proof}

\begin{lem}
For any $j\pm 1\neq i\neq j$ we have $f_{\tau_i}f_{\tau_j}=f_{\tau_j}f_{\tau_i}$.
\end{lem}
\begin{proof}
Without loss of generality, suppose that $i>j$ and $i\neq j+1$. First, define $\mathbf{r}_t$ by
\begin{align*}
\left(\begin{array}{cccccccccccccc}
x_{1} & \cdots & x_{j+1}  & x_{j+2} & x_{j+3} & \cdots x_i & x_{i+1} & x_{i+2} & x_{i+3} & x_{i+4} & \cdots \\
x_{1} & \cdots & x_{j+1}  & x_{j} & x_{j+2} & \cdots x_{i-1} & x_i & x_{i+1} & x_i & x_{i+2} & \cdots \end{array}\right. & \\
\left.\begin{array}{ccccc}
\cdots & x_{r+2} & x_{r+3} & \cdots & x_n\\
\cdots & x_{r} & x_{1} & \cdots & x_1\end{array}\right). &
\end{align*}
Note that if $i=j+2$ then the section from $j+3$ to $i$ is empty. Let
$\lambda= \neg (j,i+1)$ and $\mu=\neg(j,i+3)$, so that $\mathbf{p}_{\lambda t}=\tau_i\tau_j$ and $\mathbf{p}_{\mu
t}=\tau_j.$ Next define $\mathbf{r}_s$ by
\begin{align*}
\left(\begin{array}{cccccccccccccc}
x_{1} & \cdots & x_{j}  & x_{j+1} & x_{j+2} & \cdots x_i & x_{i+1} & x_{i+2} & x_{i+3} & x_{i+4} & \cdots \\
x_{1} & \cdots & x_{j}  &  x_{j} & x_{j+1} & \cdots x_{i-1} & x_i & x_{i+1} & x_i & x_{i+2} & \cdots\end{array}\right. & \\
\left.\begin{array}{ccccc}
\cdots & x_{r+2} & x_{r+3} & \cdots & x_n\\
\cdots & x_{r} & x_{1} & \cdots & x_1\end{array}\right). &
\end{align*}
Then $\mathbf{p}_{\lambda s}=\tau_i$ and $\mathbf{p}_{\mu s}=\epsilon$.
Thus we have
$$\left(\begin{array}{cc}
\mathbf{p}_{\lambda s} & \mathbf{p}_{\lambda t}\\
\mathbf{p}_{\mu s} & \mathbf{p}_{\mu t}\end{array}\right)=\left(\begin{array}{cc}
\tau_i & \tau_i\tau_j\\
\epsilon & \tau_j \end{array}\right)$$ implying $f_{\tau_i\tau_j}=f_{\tau_j}f_{\tau_i}.$

To complete the proof, we define $\mathbf{r}_l$ by
\begin{align*}
\left(\begin{array}{cccccccccccccc}
x_{1} & \cdots & x_{j+1} & x_{j+2} & x_{j+3} & \cdots x_i & x_{i+1} & x_{i+2} & x_{i+3} & x_{i+4} & \cdots \\
x_{1} & \cdots & x_{j+1} & x_{j} & x_{j+2} & \cdots x_{i-1} & x_i & x_{i} & x_{i+1} & x_{i+2} & \cdots \end{array}\right. & \\
\left.\begin{array}{ccccc}
\cdots & x_{r+2} & x_{r+3} & \cdots & x_n\\
\cdots & x_{r} & x_{1} & \cdots & x_1\end{array}\right).
\end{align*}
Then $\mathbf{p}_{\lambda l}=\tau_j.$ Put $\eta=\neg (j+2,i+1)$. Then
$\mathbf{p}_{\eta l}=\epsilon$ and $\mathbf{p}_{\eta t}=\tau_i,$ so $$\left(\begin{array}{cc}
\mathbf{p}_{\lambda l} & \mathbf{p}_{\lambda t}\\
\mathbf{p}_{\eta l} & \mathbf{p}_{\eta t}\end{array}\right)=\left(\begin{array}{cc}
\tau_j & \tau_j\tau_i\\
\epsilon & \tau_i \end{array}\right)$$  which implies $f_{\tau_j\tau_i}=f_{\tau_i}f_{\tau_j}$, and hence $f_{\tau_j}f_{\tau_i}=f_{\tau_i}f_{\tau_j}.$
\end{proof}

\begin{lem}
For any $i\in [1,r-2]$ we have $f_{\tau_i}f_{\tau_{i+1}}f_{\tau_i}=f_{\tau_{i+1}}f_{\tau_i}f_{\tau_{i+1}}.$
\end{lem}
\begin{proof}
Let $\rho=\tau_{i+1}\tau_i=(i\ i+1\ i+2)$ so that $\rho^2=(i\ i+2\ i+1).$

First, we show that $f_{\rho^2}=f_{\rho}f_{\rho}.$ For this purpose, we define $\mathbf{r}_j$ by
$$\left(\begin{array}{cccccccccccccccc}
x_{1} & \cdots & x_i & x_{i+1} & x_{i+2}& x_{i+3} & x_{i+4} & x_{i+5} & \cdots & x_{r+2} & x_{r+3} & \cdots & x_n\\
x_{1} & \cdots & x_i &  x_{i+1} & x_{i+2} & x_{i} & x_{i+1} & x_{i+3} & \cdots &  x_{r} & x_{1} & \cdots &
x_1\end{array}\right).$$ Let $\lambda=\neg(i,i+1)$ and $\mu=\neg(i,i+4)$, so that  $\mathbf{p}_{\lambda j}=(i\ i+2\
i+1)=\rho^2$ and $\mathbf{p}_{\mu j}=(i\ i+1\ i+2)=\rho.$ Next we define $\mathbf{r}_l$ by
$$\left(\begin{array}{cccccccccccccccc}
x_{1} & \cdots & x_i & x_{i+1} & x_{i+2}& x_{i+3} & x_{i+4} & x_{i+5} & \cdots & x_{r+2} & x_{r+3} & \cdots & x_n\\
x_{1} & \cdots & x_i &  x_{i} & x_{i+1} & x_{i+2} & x_{i} & x_{i+3} & \cdots &  x_{r} & x_{1} & \cdots &
x_1\end{array}\right).$$ Then $\mathbf{p}_{\lambda l}=(i\ i+1\ i+2)=\rho$ and $\mathbf{p}_{\mu l}=\epsilon$, so here we
have $$\left(\begin{array}{cc}
\mathbf{p}_{\lambda l} & \mathbf{p}_{\lambda j}\\
\mathbf{p}_{\mu l} & \mathbf{p}_{\mu j}\end{array}\right)=\left(\begin{array}{cc}
\rho & \rho^2\\
\epsilon & \rho \end{array}\right)$$ Hence we have $f_{\rho^2}=f_{\rho}f_{\rho}.$

Secondly, we show that $f_\rho=f_{\tau_i}f_{\tau_{i+1}}.$ Note that $\tau_{i+1}\rho=\tau_i.$ Now we define
$\mathbf{r}_s$ by
\begin{align*}
\left(\begin{array}{cccccccccccc}
x_{1} & \cdots & x_{i-1} & x_i & x_{i+1} & x_{i+2}& x_{i+3} & x_{i+4} & x_{i+5} & \cdots\\
x_{1} & \cdots & x_{i-1} & x_i &  x_{i+1} & x_{i+2} & x_{i} & x_{i+2} & x_{i+3} & \cdots
\end{array}\right. & \\
\left.\begin{array}{ccccc}
\cdots & x_{r+2} & x_{r+3} & \cdots & x_n\\
\cdots &  x_{r} & x_{1} & \cdots & x_1\end{array}\right). &
\end{align*}
Let $\nu=\neg(i,i+2)$ and $\xi=(i,i+4)$.  Then $\mathbf{p}_{\nu s}=\tau_i$ and $\mathbf{p}_{\xi s}=\rho=(i\ i+1\ i+2).$
Next, define $\mathbf{r}_t$ by
\begin{align*}
\left(\begin{array}{cccccccccccc}
x_{1} & \cdots & x_{i-1} & x_i & x_{i+1} & x_{i+2}& x_{i+3} & x_{i+4} & x_{i+5} & \cdots \\
x_{1} & \cdots & x_{i-1} & x_i &  x_{i} & x_{i+1} & x_{i+2} & x_{i+1} & x_{i+3} & \cdots \end{array}\right. & \\
\left.\begin{array}{ccccc}
\cdots & x_{r+2} & x_{r+3} & \cdots & x_n\\
\cdots &  x_{r} & x_{1} & \cdots & x_1\end{array}\right). &
\end{align*}
Then $\mathbf{p}_{\nu t}=\tau_{i+1}$ and $\mathbf{p}_{\xi t}=\epsilon$, and so we
have
$$\left(\begin{array}{cc}
\mathbf{p}_{\nu t} & \mathbf{p}_{\nu s}\\
\mathbf{p}_{\xi t} & \mathbf{p}_{\xi s}\end{array}\right)=\left(\begin{array}{cc}
\tau_{i+1} & \tau_i\\
\epsilon & \rho \end{array}\right)$$ implying $f_{\tau_i}=f_{\rho}f_{\tau_{i+1}}$, so $f_\rho=f_{\tau_i}f_{\tau_{i+1}}$ by Lemma \ref{(i i+1)}.

Finally, we show that $f_{\rho^2}=f_{\tau_{i+1}}f_{\tau_i}.$ Note that $\rho^2=(i\ i+2\ i+1)=\tau_i\tau_{i+1}$. Define
$\mathbf{r}_u$ by
\begin{align*}
\left(\begin{array}{cccccccccccc}
x_{1} & \cdots & x_{i-1} & x_i & x_{i+1} & x_{i+2}& x_{i+3} & x_{i+4} & x_{i+5} & \cdots \\
x_{1} & \cdots & x_{i-1} & x_i &  x_{i+1} & x_{i+2} & x_{i} & x_{i+1} & x_{i+3} & \cdots \end{array}\right. & \\
\left.\begin{array}{ccccc}
\cdots & x_{r+2} & x_{r+3} & \cdots & x_n\\
\cdots &  x_{r} & x_{1} & \cdots & x_1\end{array}\right). &
\end{align*}
Let $\tau=\neg(i,i+1)$ and $\delta=\neg(i+1,i+3)$. Then $\mathbf{p}_{\tau u}=\rho^2$
and $\mathbf{p}_{\delta u}=\tau_{i+1}$. Define $\mathbf{r}_v$ by $$\left(\begin{array}{cccccccccccccccc}
x_{1} & \cdots & x_{i-1} & x_i & x_{i+1} & x_{i+2}& x_{i+3} & x_{i+4} &  \cdots & x_{r+2} & x_{r+3} & \cdots & x_n\\
x_{1} & \cdots & x_{i-1} & x_i &  x_{i+1} & x_{i+1} & x_{i} & x_{i+2} & \cdots &  x_{r} & x_{1} & \cdots &
x_1\end{array}\right).$$ Then $\mathbf{p}_{\tau v}=\tau_i$ and $\mathbf{p}_{\delta v}=\epsilon$, so we have
$$\left(\begin{array}{cc}
\mathbf{p}_{\tau v} & \mathbf{p}_{\tau u}\\
\mathbf{p}_{\delta v} & \mathbf{p}_{\delta u}\end{array}\right)=\left(\begin{array}{cc}
\tau_i & \rho^2\\
\epsilon & \tau_{i+1}\end{array}\right)$$ Hence $f_{\rho^2}=f_{\tau_{i+1}}f_{\tau_i}$. We now calculate:
$$
f_{\tau_i}f_{\tau_{i+1}}f_{\tau_i}=f_{\tau_i}f_{\rho^2}=
f_{\tau_i}f_{\rho}f_{\rho}=f_{\tau_i}f_{\tau_i}f_{\tau_{i+1}}f_{\tau_i}f_{\tau_{i+1}}=f_{\tau_{i+1}}f_{\tau_i}f_{\tau_{i+1}},
$$
the final step using Lemma~\ref{(i i+1)}.
\end{proof}

We warn the reader that the relation we find below is a twist on that in (W5).

\begin{lem}
For all $i\in [1,r],$ $a,b\in G$, $f_{\iota_{b,i}}f_{\iota_{a,i}}=f_{\iota_{ab,i}},$ and so $f_{\iota_{a,i}}^{-1}=f_{\iota_{a^{-1},i}}$.
\end{lem}
\begin{proof}
Define $\mathbf{r}_j$ by
$$\left(\begin{array}{cccccccccccccccc}
x_{1} & \cdots & x_{i-1} & x_i & x_{i+1} & x_{i+2}& x_{i+3} & \cdots & x_{r+2} & x_{r+3} & \cdots & x_n\\
x_{1} & \cdots & x_{i-1} & x_i &  bx_{i} & abx_{i} & x_{i+1} &  \cdots &  x_{r} & x_{1} & \cdots &
x_1\end{array}\right).$$ Let $\lambda=\neg(i,i+2)$ and $\mu=\neg(i,i+1)$, then $\mathbf{p}_{\lambda j}=\iota_{b,i}$ and
$\mathbf{p}_{\mu j}=\iota_{ab,i}.$ Next, we define $\mathbf{r}_t$ by $$\left(\begin{array}{cccccccccccccccc}
x_{1} & \cdots & x_{i-1} & x_i & x_{i+1} & x_{i+2}& x_{i+3} & \cdots & x_{r+2} & x_{r+3} & \cdots & x_n\\
x_{1} & \cdots & x_{i-1} & x_i &  x_{i} & ax_{i} & x_{i+1} &  \cdots &  x_{r} & x_{1} & \cdots &
x_1\end{array}\right).$$ Then $\mathbf{p}_{\lambda t}=\epsilon$ and $\mathbf{p}_{\mu t}=\iota_{a,i}$, so we have
$$\left(\begin{array}{cc}
\mathbf{p}_{\mu t} & \mathbf{p}_{\mu j}\\
\mathbf{p}_{\lambda t} & \mathbf{p}_{\lambda j}\end{array}\right)=\left(\begin{array}{cc}
\iota_{a,i} & \iota_{ab,i}\\
\epsilon & \iota_{b,i} \end{array}\right)$$ implying $f_{\iota_{ab,i}}=f_{\iota_{b,i}}f_{\iota_{a,i}}.$
\end{proof}

\begin{lem} For all $i\neq j$ and $a,b\in G$ we have
$f_{\iota_{a,i}}f_{\iota_{b,j}}=f_{\iota_{b,j}}f_{\iota_{a,i}}$.
\end{lem}
\begin{proof} Without loss of generality, suppose that
 $i>j$. Recall that ${\iota_{a,i}}{\iota_{b,j}}={\iota_{b,j}}{\iota_{a,i}}$.  First define $\mathbf{r}_t$ by
\begin{align*}
\left(\begin{array}{cccccccccccc}
x_{1} & \cdots & x_{j-1} & x_j & x_{j+1} & x_{j+2} & \cdots & x_{i+1} & x_{i+2} & x_{i+3} & \cdots \\
x_{1} & \cdots & x_{j-1} & x_j &  bx_{j} & x_{j+1} & \cdots & x_i & ax_{i} & x_{i+1} & \cdots \end{array}\right. & \\
\left.\begin{array}{ccccc}
\cdots & x_{r+2} & x_{r+3} & \cdots & x_n\\
\cdots &  x_{r} & x_{1} & \cdots & x_1\end{array}\right).
\end{align*}
Let $\lambda=\neg(j,i+1)$ and $\mu=\neg(j,i+2)$.  Then $\mathbf{p}_{\lambda
t}=\iota_{a,i}\iota_{b,j}$ and $\mathbf{p}_{\mu t}=\iota_{b,j}.$

Next, we define $\mathbf{r}_s$ by
\begin{align*}
\left(\begin{array}{cccccccccccc}
x_{1} & \cdots & x_{j-1} & x_j & x_{j+1} & x_{j+2} & \cdots & x_{i+1} & x_{i+2} & x_{i+3} & \cdots \\
x_{1} & \cdots & x_{j-1} & x_j &  x_{j} & x_{j+1} & \cdots & x_i & ax_{i} & x_{i+1} & \cdots \end{array}\right. & \\
\left.\begin{array}{ccccc}
\cdots & x_{r+2} & x_{r+3} & \cdots & x_n\\
\cdots &  x_{r} & x_{1} & \cdots & x_1\end{array}\right).
\end{align*}
Then $\mathbf{p}_{\lambda s}=\iota_{a,i}$ and $\mathbf{p}_{\mu s}=\epsilon$. Thus
we have
$$\left(\begin{array}{cc}
\mathbf{p}_{\lambda s} & \mathbf{p}_{\lambda t}\\
\mathbf{p}_{\mu s} & \mathbf{p}_{\mu t}\end{array}\right)=\left(\begin{array}{cc}
\iota_{a,i} & \iota_{a,i}\iota_{b,j}\\
\epsilon & \iota_{b,j} \end{array}\right)$$ implying $f_{\iota_{b,j}}f_{\iota_{a,i}}=f_{\iota_{a,i}\iota_{b,j}}.$

Define $\mathbf{r}_l$ by
\begin{align*}
\left(\begin{array}{cccccccccccc}
x_{1} & \cdots & x_{j-1} & x_j & x_{j+1} & x_{j+2} & \cdots & x_{i+1} & x_{i+2} & x_{i+3} & \cdots \\
x_{1} & \cdots & x_{j-1} & x_j &  bx_{j} & x_{j+1} & \cdots & x_i & x_{i} & x_{i+1} & \cdots \end{array}\right. & \\
\left.\begin{array}{ccccc}
\cdots & x_{r+2} & x_{r+3} & \cdots & x_n\\
\cdots &  x_{r} & x_{1} & \cdots & x_1\end{array}\right).
\end{align*}
Then $\mathbf{p}_{\lambda l}=\iota_{b,j}.$ On the other hand, by putting
$\eta=\neg(j+1,i+1)$ we have $\mathbf{p}_{\eta l}=\epsilon$ and $\mathbf{p}_{\eta t}=\iota_{a,i}$, and so
$$\left(\begin{array}{cc}
\mathbf{p}_{\lambda l} & \mathbf{p}_{\lambda t}\\
\mathbf{p}_{\eta l} & \mathbf{p}_{\eta t}\end{array}\right)=\left(\begin{array}{cc}
\iota_{b,j} & \iota_{b,j}\iota_{a,i}\\
\epsilon & \iota_{a,i} \end{array}\right)$$ which implies $f_{\iota_{b,j}\iota_{a,i}}=f_{\iota_{a,i}}f_{\iota_{b,j}}$, and hence $f_{\iota_{a,i}}f_{\iota_{b,j}}=f_{\iota_{b,j}}f_{\iota_{a,i}}.$
\end{proof}

\begin{lem} For any $i,j$ with $i\neq j, j+1$ and $a\in G$ we have
$f_{\iota_{a,i}}f_{\tau_j}=f_{\tau_j}f_{\iota_{a,i}}.$
\end{lem}
\begin{proof}
Suppose that $i<j$; the proof for $j<i$ is entirely similar. Then
\begin{align*}
\iota_{a,i}\tau_j=\left(\begin{array}{ccccccccccccccc}
x_{1} & \cdots & x_{i-1} & x_i & x_{i+1} & \cdots & x_{j-1} & x_j & x_{j+1}  \\
x_{1} & \cdots & x_{i-1} & ax_i &  x_{i+1} & \cdots & x_{j-1} & x_{j+1}& x_j
\end{array}\right. & \\
\left.\begin{array}{ccc}
x_{j+2} & \cdots & x_r \\
x_{j+2} & \cdots & x_r
\end{array}\right). &
\end{align*}
Define $\mathbf{r}_t$ by
\begin{align*}
\left(\begin{array}{cccccccccccccccc}
x_{1} & \cdots & x_{i-1} & x_i & x_{i+1} & x_{i+2} & \cdots & x_{j} & x_{j+1} & x_{j+2} & x_{j+3} & x_{j+4} & \cdots \\
x_{1} & \cdots & x_{i-1} & x_i & ax_{i} & x_{i+1} & \cdots & x_{j-1} & x_{j} & x_{j+1} & x_j & x_{j+2} & \cdots
\end{array}\right. & \\
\left.\begin{array}{ccccc}
\cdots & x_{r+2} & x_{r+3} & \cdots & x_n\\
\cdots & x_{r} & x_{1} & \cdots & x_1\end{array}\right). &
\end{align*}
Let $\lambda=\neg(i,j+1)$ and $\mu=\neg(i+1,j+1)$. Then
$\mathbf{p}_{\lambda t}=\iota_{a,i}\tau_j$ and $\mathbf{p}_{\mu t}=\tau_j.$

Define $\mathbf{r}_s$ by
\begin{align*}
\left(\begin{array}{cccccccccccccccc}
x_{1} & \cdots & x_{i-1} & x_i & x_{i+1} & x_{i+2} & \cdots & x_{j} & x_{j+1} & x_{j+2} & x_{j+3} & x_{j+4} & \cdots \\
x_{1} & \cdots & x_{i-1} & x_i & ax_{i} & x_{i+1} & \cdots & x_{j-1} & x_{j} & x_{j} & x_{j+1} & x_{j+2} & \cdots
\end{array}\right. & \\
\left.\begin{array}{ccccc}
\cdots & x_{r+2} & x_{r+3} & \cdots & x_n\\
\cdots & x_{r} & x_{1} & \cdots & x_1\end{array}\right). &
\end{align*}
Then $\mathbf{p}_{\lambda s}=\iota_{a,i}$ and
$\mathbf{p}_{\mu s}=\epsilon$. Hence we have $$\left(\begin{array}{cc}
\mathbf{p}_{\lambda s} & \mathbf{p}_{\lambda t}\\
\mathbf{p}_{\mu s} & \mathbf{p}_{\mu t}\end{array}\right)=\left(\begin{array}{cc}
\iota_{a,i} & \iota_{a,i}\tau_j\\
\epsilon & \tau_j \end{array}\right)$$ implying $f_{\iota_{a,i}\tau_j}=f_{\tau_j}f_{\iota_{a,i}}.$

Next we define $\eta=\neg(i,j+3)$, so that  $\mathbf{p}_{\eta t}=\iota_{a,i}$. Now let $\mathbf{r}_l$ be
\begin{align*}
\left(\begin{array}{cccccccccccccccc}
x_{1} & \cdots & x_{i-1} & x_i & x_{i+1} & x_{i+2} & \cdots & x_{j} & x_{j+1} & x_{j+2} & x_{j+3} & x_{j+4} & \cdots \\
x_{1} & \cdots & x_{i-1} & x_i & x_{i} & x_{i+1} & \cdots & x_{j-1} & x_{j} & x_{j+1} & x_j & x_{j+2} & \cdots
\end{array}\right. & \\
\left.\begin{array}{ccccc}
\cdots & x_{r+2} & x_{r+3} & \cdots & x_n\\
\cdots & x_{r} & x_{1} & \cdots & x_1\end{array}\right). &
\end{align*}
Then $\mathbf{p}_{\lambda l}=\tau_j$ and $\mathbf{p}_{\eta
l}=\epsilon$, so
$$\left(\begin{array}{cc}
\mathbf{p}_{\lambda l} & \mathbf{p}_{\lambda t}\\
\mathbf{p}_{\eta l} & \mathbf{p}_{\eta t}\end{array}\right)=\left(\begin{array}{cc}
\tau_j & \tau_j\iota_{a,i}\\
\epsilon & \iota_{a,i} \end{array}\right)$$ implying $f_{\tau_j\iota_{a,i}}=f_{\iota_{a,i}}f_{\tau_j}$, so $f_{\tau_j}f_{\iota_{a,i}}=f_{\iota_{a,i}}f_{\tau_j}.$
\end{proof}

\begin{lem} For any $i\in [1,r-1]$ and $a\in G$ we have
$f_{\iota_{a,i}}f_{\tau_i}=f_{\tau_i}f_{\iota_{a,i+1}}.$
\end{lem}
\begin{proof} We have
$$\iota_{a,i}\tau_i=\left(\begin{array}{cccccccccccccccccc}
x_{1} & \cdots & x_{i-1} & x_i & x_{i+1} & x_{i+2} & \cdots & x_r \\
x_{1} & \cdots & x_{i-1} & ax_{i+1} &  x_{i} & x_{i+2} & \cdots & x_r \end{array}\right)={\tau_i}{\iota_{a,i+1}}.$$
Define $\mathbf{r}_t$ by
\begin{align*}
\left(\begin{array}{ccccccccccccccc}
x_{1} & \cdots & x_{i-1} & x_i & x_{i+1} & x_{i+2} & x_{i+3} & x_{i+4} & \cdots & x_{r+2}  \\
x_{1} & \cdots & x_{i-1} & x_{i} &  x_{i+1} & ax_{i+1} & x_i & x_{i+2} & \cdots & x_r
\end{array}\right. & \\
\left.\begin{array}{ccc}
x_{r+3} & \cdots & x_n \\
x_1 & \cdots & x_1\end{array}\right).
\end{align*}
Define $\lambda=\neg(i,i+1)$ and $\mu=\neg(i,i+2)$. Then $\mathbf{p}_{\lambda t}=\iota_{a,i}\tau_i$ and
$\mathbf{p}_{\mu t}=\tau_i.$ Define $\mathbf{r}_s$ by
$$\left(\begin{array}{cccccccccccccccccc}
x_{1} & \cdots & x_{i-1} & x_i & x_{i+1} & x_{i+2} & x_{i+3} & \cdots & x_{r+2} & x_{r+3} & \cdots & x_n \\
x_{1} & \cdots & x_{i-1} & x_{i} &  x_{i} & ax_{i} & x_{i+1} & \cdots & x_r & x_1 & \cdots & x_1\end{array}\right).$$
Then $\mathbf{p}_{\lambda s}=\iota_{a,i}$ and $\mathbf{p}_{\mu s}=\epsilon$, so we have $$\left(\begin{array}{cc}
\mathbf{p}_{\lambda s} & \mathbf{p}_{\lambda t}\\
\mathbf{p}_{\mu s} & \mathbf{p}_{\mu t}\end{array}\right)=\left(\begin{array}{cc}
\iota_{a,i} & \iota_{a,i}\tau_i\\
\epsilon & \tau_i \end{array}\right)$$ so $f_{\iota_{a,i}\tau_i}=f_{\tau_i}f_{\iota_{a,i}}.$

Now put $\eta=\neg(i+1,i+3)$, so that $\mathbf{p}_{\eta t}=\iota_{a,i+1}.$ Define $\mathbf{r}_l$ by
$$\left(\begin{array}{cccccccccccccccccc}
x_{1} & \cdots & x_{i-1} & x_i & x_{i+1} & x_{i+2} & x_{i+3} & x_{i+4} & \cdots & x_{r+2} & x_{r+3} & \cdots & x_n \\
x_{1} & \cdots & x_{i-1} & x_{i} &  x_{i} & x_{i+1} & x_{i} & x_{i+2} & \cdots & x_r & x_1 & \cdots &
x_1\end{array}\right).$$ Then $\mathbf{p}_{\lambda l}=\tau_i$ and $\mathbf{p}_{\eta l}=\epsilon$, so
$$\left(\begin{array}{cc}
\mathbf{p}_{\lambda l} & \mathbf{p}_{\lambda t}\\
\mathbf{p}_{\eta l} & \mathbf{p}_{\eta t}\end{array}\right)=\left(\begin{array}{cc}
\tau_i & \tau_i\iota_{a,i+1}\\
\epsilon & \iota_{a,i+1} \end{array}\right)$$ so that $f_{\tau_i\iota_{a,i+1}}=f_{\iota_{a,i+1}}f_{\tau_i}$. Thus
$f_{\tau_i}f_{\iota_{a,i}}=f_{\iota_{a,i+1}}f_{\tau_i}$ and so
$f_{\iota_{a,i}}f_{\tau_i}=f_{\tau_i}f_{\iota_{a,i+1}},$ bearing in mind Lemmas 9.5 and 9.8.
\end{proof}

We denote by $\Omega$ all the following relations we have obtained so far on the set of generators $$T=\{f_{\tau_i},
f_{\iota_{a,j}}: ~1\leq i\leq r-1, 1\leq j\leq r, a\in G\}$$ of $ \overline{\overline{H}}$:

$(T1)$ $f_{\tau_i}f_{\tau_i}=1$, $1\leq i\leq r-1.$

$(T2)$ $f_{\tau_i}f_{\tau_j}=f_{\tau_j}f_{\tau_i}$, $j\pm 1\neq i\neq j$.

$(T3)$ $f_{\tau_i}f_{\tau_{i+1}}f_{\tau_i}=f_{\tau_{i+1}}f_{\tau_i}f_{\tau_{i+1}}$, $1\leq i\leq r-2.$

$(T4)$ $f_{\iota_{a,i}}f_{\iota_{b,j}}=f_{\iota_{b,j}}f_{\iota_{a,i}}$, $a,b\in G$ and $1\leq i\neq j\leq r.$

$(T5)$ $f_{\iota_{b,i}}f_{\iota_{a,i}}=f_{\iota_{ab,i}}$, $1\leq i\leq r$ and $a,b\in G.$

$(T6)$ $f_{\iota_{a,i}}f_{\tau_j}=f_{\tau_j}f_{\iota_{a,i}},$ $1\leq i\neq j, j+1\leq r.$

$(T7)$ $f_{\iota_{a,i}}f_{\tau_i}=f_{\tau_i}f_{\iota_{a,i+1}},$ $1\leq i\leq r-1$ and $a\in G.$

\bigskip

Note that the relations $(T1)-(T7)$ match exactly the relations $(W1)-(W7)$, except for a twist between $(T5)$ and $(W5)$, as we mentioned before.

\bigskip

We now have all the ingredients in place to prove the following.

\begin{prop}\label{prop:main}
The group $ \overline{\overline{H}}$ with a presentation $\mathcal{Q}=\langle S: \Gamma\rangle$ of Lemma \ref{presentation} is isomorphic to the presentation $\mathcal{U}=\langle Y:~ \Upsilon \rangle$ of $H$ given in Lemma~\ref{presentation of wreath product}, so that $\overline{H}\cong H$.
\end{prop}

\begin{proof}
 We define a map ${\boldsymbol \theta}: \widetilde{Y}\longrightarrow  \overline{\overline{H}}$ by $$\tau_i{\boldsymbol \theta}=f_{\tau_i}^{-1}(=f_{\tau_i}), ~\iota_{a,j}{\boldsymbol \theta}=f_{\iota_{a,j}}^{-1}(=f_{\iota_{a^{-1},j}})$$ where $1\leq i\leq r-1,\ 1\leq j\leq r, a\in G.$ Now we claim that $\Upsilon \subseteq \ker {\boldsymbol \theta}$. Clearly, the relations corresponding to $(W1)-(W4)$ and $(W6)$ and $(W7)$ lie in  $\ker {\boldsymbol \theta}$. Moreover, considering $(W5)$ $$(\iota_{a,i}\iota_{b,i}){\boldsymbol \theta}=\iota_{a,i}{\boldsymbol \theta}\iota_{b,i}{\boldsymbol \theta}=f_{\iota_{a,i}}^{-1}f_{\iota_{b,i}}^{-1}=f_{\iota_{a^{-1},i}}f_{\iota_{b^{-1},i}}=f_{\iota_{b^{-1}a^{-1},i}}=f_{\iota_{(ab)^{-1},i}}=\iota_{ab,i}{\boldsymbol \theta}$$ so that $\Upsilon \subseteq \ker {\boldsymbol \theta}$, and hence  there exists a well defined morphism $\overline{{\boldsymbol \theta}}: H\longrightarrow  \overline{\overline{H}}$ given by $\tau_i\overline{{\boldsymbol \theta}}=f_{\tau_i}^{-1}$ and $\iota_{a,j}\overline{{\boldsymbol \theta}}=f_{\iota_{a,j}}^{-1}$, where $1\leq i\leq r-1,\ 1\leq j\leq r, a\in G.$

Conversely, we define ${\boldsymbol \psi}: \widetilde{S}\longrightarrow H$ by $f_\phi{\boldsymbol \psi}=\phi^{-1}$. We show that $\Gamma\subseteq \ker {\boldsymbol \psi}$. Clearly, $f_\epsilon{\boldsymbol \psi}=\epsilon^{-1}=\epsilon=1{\boldsymbol \psi}.$ Suppose that
 $(\phi,\varphi, \psi,\sigma)$ is singular, giving  $\phi\varphi^{-1}=\psi\sigma^{-1}$. Then $$(f_\phi^{-1}f_\varphi){\boldsymbol \psi}=(f_\phi{\boldsymbol \psi})^{-1}f_\varphi{\boldsymbol \psi}=\phi\varphi^{-1}=\psi\sigma^{-1}=(f_\psi{\boldsymbol \psi})^{-1}f_\sigma{\boldsymbol \psi}=(f_\psi^{-1}f_\sigma){\boldsymbol \psi}$$ so $\Gamma \subseteq \ker {\boldsymbol \psi}.$ Thus there exists a well defined morphism $\overline{{\boldsymbol \psi}}:  \overline{\overline{H}}\longrightarrow H$ given by $f_\phi\overline{{\boldsymbol \psi}}=\phi^{-1}.$ Then $$\tau_i\overline{{\boldsymbol \theta}}~\overline{{\boldsymbol \psi}}=f_{\tau_i}^{-1}\overline{{\boldsymbol \psi}}=(f_{\tau_i}\overline{{\boldsymbol \psi}})^{-1}=\tau_i$$ and $$\iota_{a,i}\overline{{\boldsymbol \theta}}~\overline{{\boldsymbol \psi}}=f_{\iota_{a,i}}^{-1}\overline{{\boldsymbol \psi}}=(f_{\iota_{a,i}}\overline{{\boldsymbol \psi}})^{-1}=\iota_{a,i}$$ hence $\overline{{\boldsymbol \theta}}~\overline{{\boldsymbol \psi}}$ is the identity mapping, and  so $\overline{{\boldsymbol \theta}}$ is one-one.
 Since $T$ is a set of generators for $\overline{\overline{H}}$, it is clear that $\overline{{\boldsymbol \theta}}$ is onto, and so $$\overline{H}\cong \overline{\overline{H}}\cong H \cong G\wr S_r.$$
\end{proof}

We can now state the  main theorem of this paper.

\begin{them}\label{thm:n-2}
Let $\en F_n(G)$ be the endomorphism monoid of a free $G$-act $F_n(G)$ on $n$ generators, where $n\in \mathbb{N}$ and $n\geq 3$, let  $E$ be the biordered set of idempotents of $\en F_n(G)$, and  let $\ig(E)$ be the free idempotent generated semigroup over $E$.

For any idempotent $\epsilon\in E$ with rank $r$, where $1\leq r\leq n-2$, the maximal subgroup $\overline{H}$ of $\ig(E)$ containing $\overline{\epsilon}$ is isomorphic to  the maximal subgroup $H$ of $\en F_n(G)$ containing $\epsilon$  and hence to $G\wr\mathcal{S}_r$.
\end{them}

Note that if $\epsilon$ is an idempotent with rank $n$, that is, the identity map, then $\overline{H}$ is the trivial group, since it is generated (in $\ig(E)$) by idempotents of the same rank. On the other hand, if
the rank of $\epsilon$ is $n-1$,  then $\overline{H}$ is the free group as there are no non-trivial singular squares in the $\mathcal{D}$-class of $\epsilon$ in $\en F_n(G).$

\bigskip

If $\e$ is a rank 1 idempotent, then $\overline{H}\cong H\cong G,$ so that we re-obtain the main result of \cite{gray:2012}, \cite{gouldyang:2013} and \cite{dolinkaRuskuc:2013} by yet another method.

\begin{coro}\cite{gray:2012, gouldyang:2013, dolinkaRuskuc:2013}
Every group is a maximal subgroup of $\ig(E)$, for some $E.$
\end{coro}

Finally, if $G$ is trivial, then $\en F_n(G)$ is essentially $\mathcal{T}_n$, so we deduce the following result from
\cite{gray:2012a}.

\begin{coro}\cite{gray:2012a}
 Let $n\in \mathbb{N}$ with $n\geq 3$ and let  $\ig(E)$ be the free idempotent generated semigroup over the biordered set $E$ of idempotents of $\mathcal{T}_n$.

For any idempotent $\epsilon\in E$ with rank $r$, where $1\leq r\leq n-2$, the maximal subgroup $\overline{H}$ of $\ig(E)$ containing $\overline{\epsilon}$ is isomorphic to  the maximal subgroup $H$ of $\mathcal{T}_n$ containing $\epsilon,$  and hence to $\mathcal{S}_r$.

\end{coro}

\begin{acknowledgements}
The authors would like to thank Robert D.\ Gray (University of East Anglia) and Nik Ru\v{s}kuc (University of
St Andrews) for some useful discussions, and John Fountain (University of York) for suggesting the problem that led to this work.
\end{acknowledgements}


\end{document}